\documentclass[11pt,reqno]{amsart}
\usepackage{amssymb, bbm}
\usepackage[dvipsnames]{xcolor}
\usepackage[makeroom]{cancel} 
\usepackage{dsfont}
\usepackage{amsfonts}
\usepackage{mathrsfs}
\usepackage{amsmath}
\usepackage{graphicx}
\usepackage{hyperref}
\usepackage{float}
\usepackage{epstopdf}
\usepackage{color}
\usepackage{bm}
\usepackage{comment}
\usepackage{soul}
\usepackage{esint}
\usepackage[shortlabels]{enumitem}
\usepackage{eufrak}
\usepackage{tikz}

\numberwithin{equation}{section}

\title{}
\author{}
\begin{document}

\newtheorem{lem}{Lemma}[section]
\newtheorem{prop}[lem]{Proposition}
\newtheorem{cor}[lem]{Corollary}
\newtheorem{thm}[lem]{Theorem}
\newtheorem{rmk}[lem]{Remark}
\theoremstyle{definition}
\newtheorem{deff}[lem]{Definition}

\def\mcE{\mathcal{E}}
\def\mcF{\mathcal{F}}
\def\tmcE{\tilde{\mathcal{E}}}
\def\omcE{\overline{\mathcal{E}}}
\def\tmcF{\tilde{\mathcal{F}}} 
\def\red{\color{red}}
\def\omcE{\overline{\mathcal{E}}}
\def\R{\mathbb{R}}
\def\Z{\mathbb{Z}}
\def\tF{\tilde{F}}
\def\Tr{\mathrm{Tr}}
\def\mcH{\mathcal{H}}
\def\V{\mathbb{V}}
\def\res{\mathrm{res}}
\def\osc{\mathrm{Osc}}

\title[A non-canonical diffusion on the Sierpi\'nski carpet]{A non-canonical diffusion on the Sierpi\'nski carpet}

\author{Shiping Cao}
\address{Department of Mathematics, The Chinese University of Hong Kong, Shatin, Hong Kong}
\email{spcao@math.cuhk.edu.hk}
\thanks{}

\author{Hua Qiu}
\address{School of Mathematics, Nanjing University, Nanjing, 210093, P. R. China.}
\thanks{The research of Qiu was supported by the National Natural Science Foundation of China, grant 12471087 and 12531004.}
\email{huaqiu@nju.edu.cn}

\author{Bingshen Wang}
\address{School of Mathematics, Nanjing University, Nanjing, 210093, P. R. China.} 
\email{652022210009@smail.nju.edu.cn}


\date{}


\maketitle

\begin{abstract}
We constructed a diffusion process on the Sierpi\'nski carpet that satisfies the sub-Gaussian heat kernel estimate with respect to the Euclidean metric and a non-standard self-similar measure. 
\end{abstract}

\section{Introduction}
The construction of Brownian motions on the Sierpi\'nski gasket \cite{BPgasket,Goldstein,Kusuoka} and on generalized Sierpi\'nski carpets \cite{barlow1989construction} marked the beginning of analysis on fractals. Since then, the properties of such processes have been extensively investigated, and various extensions have been developed. 

\begin{figure}[htp]
	\includegraphics[width=4.40cm]{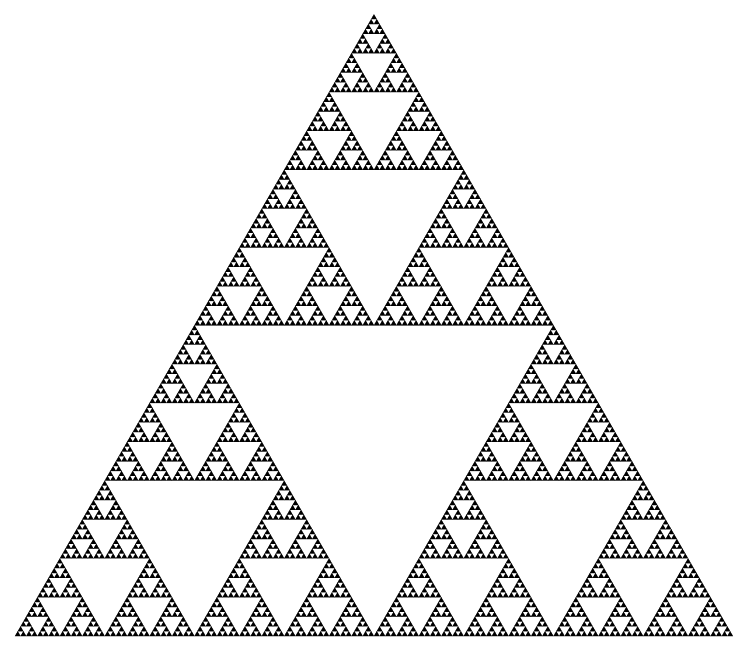}\qquad 
	\includegraphics[width=3.9cm]{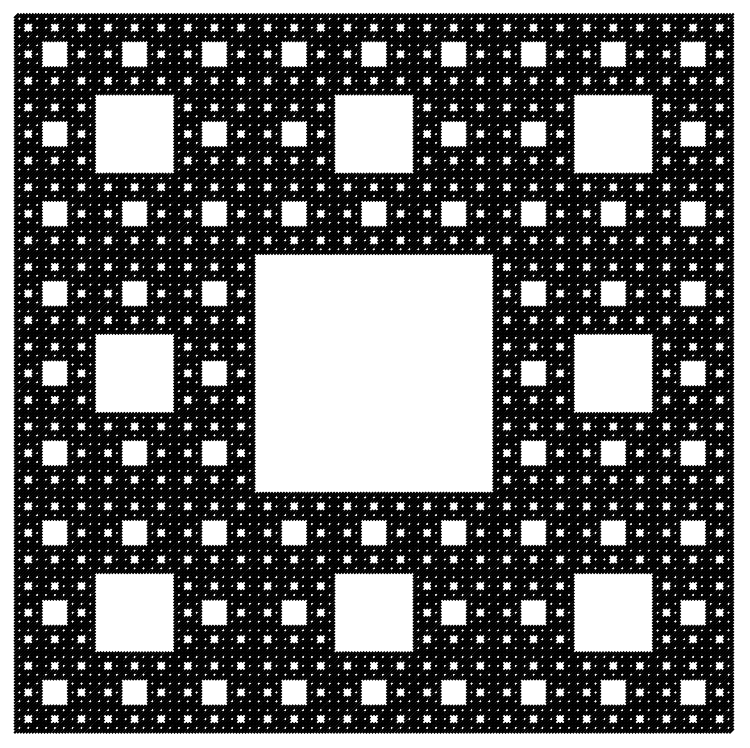}
	\caption{The standard Sierpi\'nski gasket and carpet.}
\end{figure}

In 1993, Kigami \cite{Kigamipcf} introduced the class of post-critically finite (p.c.f.) self-similar sets---a natural generalization of the Sierpi\'nski gasket---and established a comprehensive framework for constructing self-similar Dirichlet forms on such sets (see also his monograph \cite{Kigamipcf}). It is worth noting that Dirichlet forms and diffusion processes are essentially equivalent: every regular Dirichlet form corresponds to a Hunt process, and vice versa. 

A notable feature of Kigami's framework on p.c.f. self-similar sets is the flexibility to assign different weights to cells when defining the self-similar structure. For instance, Sabot \cite[Example 1.2]{Sabot} characterized the set of admissible weights on the Sierpi\'nski gasket that yield a self-similar Dirichlet form. Building on this, Hambly and Kumagai \cite{HamblyKumagai} derived heat kernel estimates for the associated Dirichlet forms on p.c.f.  sets under suitably chosen metrics and measures compatible with the weights.\smallskip 

In contrast to p.c.f. sets, generalized Sierpi\'nski carpets lack local cut points, and have seen fewer generalizations. In 1992, Kusuoka and Zhou \cite{KusuokaZhou} proposed a framework for constructing self-similar Dirichlet forms on certain fractals and  proved the existence of such a form on a generalized Sierpi\'nski carpet using the Knight move argument of Barlow and Bass \cite{barlow1989construction}. In 1999, Barlow and Bass \cite{BB19993d} extended this construction to higher-dimensional generalized Sierpi\'nski carpets.  More than two decades later,  two of the authors \cite{CaoQiu} introduced unconstrained Sierpi\'nski carpets---extension of generalized Sierpi\'nski carpets that allow cells to exist off-grids---and, using the cell graph approximation framework of Kusuoka and Zhou \cite{KusuokaZhou}, constructed self-similar Dirichlet forms on them. This approach was later extended to certain hollow carpets \cite{CQW}. Inspired by the ideas of Kusuoka and Zhou \cite{KusuokaZhou}, Kigami \cite{Kigamiphomogenity}, together with  Kigami and Ota \cite{KO},  also constructed $p$-energies on polygon-based self-similar sets.

A key property of diffusion processes on Sierpi\'nski carpet like fractals is the sub-Gaussian heat kernel estimate
\begin{equation*}
	\begin{split} 
		C_1\frac{1}{|x-y|^{\alpha/\beta}}\exp(-C_2(\frac{|x-y|^{\beta}}{t})^{\frac{1}{\beta-1}})&\leq p_t(x,y)\\&\leq C_3\frac{1}{|x-y|^{\alpha/\beta}}\exp(-C_4(\frac{|x-y|^{\beta}}{t})^{\frac{1}{\beta-1}}),
	\end{split} 
\end{equation*} 
 established by Barlow and Bass \cite{BB1992SCHKestimate,BB19993d} for generalized Sierpi\'nski carpets. Here, $\alpha$ is the Hausdorff dimension of the fractal and $\beta$ is the walk dimension. 

A natural question is whether one can assign different weights to cells and define other diffusion processes on the Sierpi\'nski carpet that also satisfy  sub-Gaussian heat kernel estimates, as is possible on many p.c.f. self-similar sets.   It is worth emphasizing that previous constructions of Dirichlet forms on carpet-type fractals all assumed local symmetry. In contrast, the present work does not rely on such symmetry. 

In this paper, we address a simple case of the question. We consider a (non-standard) self-similar measure $\mu$ on the standard Sierpi\'nski carpet that possesses full symmetry, and construct a diffusion process that satisfies the corresponding sub-Gaussian heat kernel estimates. See Section \ref{section2} and Figure \ref{fig2} for the precise definition of $\mu$.

\begin{figure}[htp]
	\includegraphics[width=3.9cm]{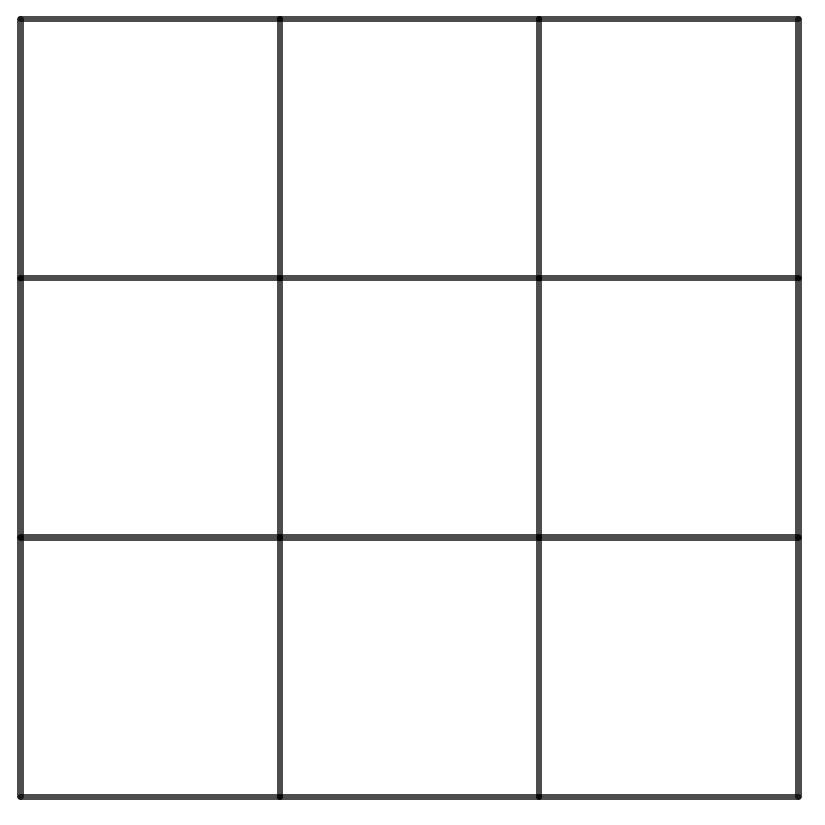}
		\caption{Self-similar weight of $\mu$ with $\rho>0$.}
		\begin{picture}(0,0)
		\put(-46,57){$\frac{1}{4+4\rho}$}\put(-11,57){$\frac{\rho}{4+4\rho}$}\put(24,57){$\frac{1}{4+4\rho}$}\put(-46,126){$\frac{1}{4+4\rho}$}\put(-11,126){$\frac{\rho}{4+4\rho}$}\put(24,126){$\frac{1}{4+4\rho}$}
		\put(-46,91){$\frac{\rho}{4+4\rho}$}\put(24,91){$\frac{\rho}{4+4\rho}$}
	\end{picture}
	\label{fig2}
\end{figure}

\begin{thm}\label{thm1}
There is a Feller process $((X_t)_{t>0},(P^x)_{x\in F})$ on the Sierpi\'nski carpet $F$ with a transition density $q_t(x,y)$ satisfying the two-sided sub-Gaussian heat kernel estimate 
\begin{equation}\label{e:1.1}
	\begin{split} 
		C_1\frac{1}{\mu(B(x,t^{1/\beta}))}\exp(-C_2(\frac{|x-y|^{\beta}}{t})^{\frac{1}{\beta-1}})&\leq q_t(x,y)\\&\leq C_3\frac{1}{\mu(B(x,t^{1/\beta}))}\exp(-C_4(\frac{|x-y|^{\beta}}{t})^{\frac{1}{\beta-1}})
	\end{split} 
\end{equation}
for all $x,y\in F$ and $t\leq 1$. 
\end{thm} 

We conclude this section with an outline of the paper. 

Section \ref{section2} introduces the basic notation, including the definition of the Sierpi\'nski carpet and related structures. 

Section \ref{section3}, the major part of the paper, reproduces the Knight move argument of Barlow and Bass \cite{barlow1989construction}, overcoming a major difficulty arising from the presence of cells with different weights. 

Section \ref{section4} is devoted to proving the elliptic Harnack inequality. Section \ref{section5} establishes  resistance estimates. Finally, in Section \ref{section6}, we combine these results  to prove Theorem \ref{thm1}.

\section{Notation}\label{section2}
We begin by defining the Sierpi\'nski carpet and its pre-carpet.

\begin{deff}
Let $F_0=[0,1]^2$ be the unit square. For $i=1,2,\cdots,8$, define similarities $\Phi_i:\mathbb{R}^2\to \mathbb{R}^2$ as
\begin{align*}
&\Phi_1(x)=\frac{1}{3}x,\ \Phi_2(x)=\frac{1}{3}x+(\frac{1}{3},0),\ \Phi_3(x)=\frac{1}{3}x+(\frac{2}{3},0),\ \Phi_4(x)=\frac{1}{3}x+(\frac{2}{3},\frac{1}{3}),\\ &\Phi_5(x)=\frac{1}{3}x+(\frac{2}{3},\frac{2}{3}),\ \Phi_6(x)=\frac{1}{3}x+(\frac{1}{3},\frac{2}{3}),\ \Phi_7(x)=\frac{1}{3}x+(0,\frac{2}{3}),\ \Phi_8(x)=\frac{1}{3}x+(0,\frac{1}{3}).
\end{align*}

Let $F_n=\bigcup_{i=1}^8\Phi_i(F_{n-1})$ for $n\geq 1$, and let $F=\bigcap_n F_n$. We call $F$ the \textit{Sierpi\'nski carpet} (SC for abbreviation),  and call $\tilde{F}:=\bigcup_{n=0}^{\infty}3^n F$ the \textit{unbounded Sierpi\'nski carpet} (unbounded SC for abbreviation). Let $\tilde{F}_0=\bigcup_{m=0}^{\infty}3^mF_m$ be the \textit{pre-carpet}. For $n\geq 1$, denote $\tilde{F}_n=3^{-n}\tilde{F}_0$. 
\end{deff}

Let us introduce some more notations about word spaces and Dirichlet forms.
\begin{deff}
$(1)$. Let $\mathcal{S}_0=\emptyset,\ \mathcal{S}_1=\{1,2,3,4,5,6,7,8\}$ be the alphabet associated with $F$. Throughout the paper, we use $\mathcal{S}_n$ to denote the collection of \textit{words} of length $n$: $$\mathcal{S}_n=\mathcal{S}_1^n=\{w=w_1...w_n|w_i\in \mathcal{S}_1\}.$$

$(2)$. For $w\in \mathcal{S}_n$, we denote by $|w|=n$ the \textit{length} of $w$, and write 
$$\Phi_w:=\Phi_{w_1}\circ...\circ \Phi_{w_n},$$
and $F_w:=\Phi_wF_0$. Note that $F_w$ is a square in $F_n$ of side length $3^{-n}$.
\end{deff} 

In this paper, we consider non-standard Dirichlet forms on SC, whose weights of level $1$ cells will be different.  Throughout the paper, we fix a positive number $\rho$. Let 
\[
\rho_{2i}=\rho\hbox{ and }\rho_{2i-1}=1\quad\hbox{ for } 1\leq i\leq 4. 
\]
For longer words, we define 
\[
\rho_w=\rho_{w_1}\cdots \rho_{w_n}\hbox{ for }n\geq 1\hbox{ and }w\in \mathcal{S}_n. 
\]
We introduce the measures the Dirichlet forms associated with the above weights. \medskip

\noindent{(Measures). (a). Define the measure $\mu_0$ on $\tilde{F}_0$ as 
\[
\mu_0(A)=\lim\limits_{m\to\infty}\sum_{w\in \mathcal{S}_m}\rho_w\,|A\cap 3^mF_w|,
\]
for every Borel subset $A$ of $\tilde{F}_0$, where $|A|$ is the Lebesgue measure of $A$. 

(b). For $n\geq 1$, we define $\mu_n$ on $\tilde{F}_n$ as 
\[
\mu_n(A)=(\sum_{i=1}^8\rho_i)^{-n}\mu_0(3^nA),
\]
for every Borel subset $A$ of $\tilde{F}_n$.

(c). We denote by $\mu$ the vague limit of $\mu_n$ as $n\to\infty$. In particular, $\mu$ restricted on $F$ is the self-similar measure that satisfies 
\[
\mu(A)=(\sum_{i=1}^8\rho_i)^{-1}\sum_{i=1}^8\rho_i\mu(\Phi_i^{-1}(A)),
\]
for every Borel subset $A$ of $F$. \medskip


\noindent (Dirichlet form). Let 
$$
\|f\|_{\tilde{F}_0}:=\sqrt{\int_{\tilde{F}_0}|\nabla f|^2\mu_0(dx)+\int_{\tilde{F}_0}| f|^2\mu_0(dx)}\quad\hbox{ for }f\in C_c^1(\tilde{F}_0). 
$$
We define a strongly local, regular Dirichlet form $(\tmcE_0,\tmcF_0)$ on $L^2(\tilde{F}_0,\mu_0)$ as 
\begin{align*}
\tmcF_0\hbox{ is the closure of }C_c^1(\tilde{F}_0)\hbox{ in }\|f\|_{\tilde{F}_0},\\
\tmcE_0(f,g)=\int_{\tilde{F}_0}\nabla f(x)\cdot\nabla g(x)\mu_0(dx).
\end{align*}
Note that $\tmcF_0=W^{1,2}(\tilde{F}_0)$. Write $\tmcE_0(f)=\tmcE_0(f,f)$ for short. Our goal is to construct a Dirichlet form $L^2(F,\mu)$ as the scaling limit of $(\tmcE_0,\tmcF_0)$.}\medskip

Next, we define neighborhoods of a point $x\in\mathbb{R}^2$. Let
$$
\mathcal{Q}_n=\{[i3^{n},(i+1)3^{n}]\times [j3^{n},(j+1)3^{n}]:\ i,j\in\mathbb{Z}\}.
$$
Clearly, $\mathcal{Q}_n$ is a collection of squares whose vertices are $3^{n}$ lattices. Moreover, for $A\subset\R^2$, we write 
\[
\mathcal{Q}_n(A)=\{Q\in \mathcal{Q}_n:\hbox{int}(Q)\cap A\neq\emptyset\}. 
\]

(a). For $x\in \mathbb{R}^2$ and $n\in\mathbb{Z}$, we choose unique $i,j\in \mathbb{Z}$ so that 
\[
x\in [i3^{n},(i+1)3^{n})\times [j3^{n},(j+1)3^{n}),
\]
and let 
\[
Q_n(x):=[i3^{n},(i+1)3^{n}]\times [j3^{n},(j+1)3^{n}].
\] 

(b). For $x\in \mathbb{R}^2$ and $n\in\mathbb{Z}$, we choose unique $i,j\in \mathbb{Z}$ so that 
\[
x\in [(i-1/2)3^{n},(i+1/2)3^{n})\times[(j-1/2)3^{n},(j+1/2)3^{n}),
\]
and let 
\[
D_n(x):=[(i-1)3^{n},(i+1)3^{n}]\times [(j-1)3^{n},(j+1)3^{n}].
\]\smallskip

Finally, we introduce some notation related to random processes and stopping times. Let $X_t$ be a process with continuous path, and $A$ be a subset in the state space of $X$. 
Let \begin{equation*}
\begin{aligned}
\mathcal{T}_A(X)&=\inf\{t\geq 0:X_t\in A\}, \\
\tau_A(X)&=\inf\{t>0:X_t\notin A\},\\
\sigma_0^n(X)&=\inf\{t\geq0:X_t\in \bigcup_{Q\in \mathcal{Q}_n}\partial Q\},\\
\sigma_1^n(X)&=\inf\{t> 0:X_t\in \partial D_n(X(\sigma_0^n))\},\\
\sigma_{i+1}^n(X)&=\inf\{t>\sigma_i^n(X):X_t\in \partial D_n(X(\sigma_i^n))\}.
\end{aligned}
\end{equation*}

\section{ Knight moves}\label{section3}
Let $\tilde{F_0}$ be the pre-carpet, and let $(\tilde{\mcE}_0,\tmcF_0)$ be the Dirichlet form on $\tilde{F_0}$ defined in Section \ref{section2}. Since $(\tilde{\mcE}_0,\tmcF_0)$ is strongly local and regular, there exists a diffusion $W_t$ associated with $(\tilde{\mcE}_0,\tmcF_0)$ \cite{fukushima2011dirichlet}. Let  $m\in\mathbb{Z}$. In this section, we estimate the lower bound of hitting distribution on $\partial D_m(W(\sigma_i^m))$.  Let 
\[
 \mathcal{H}_m=\bigcup\limits_{Q,Q^{\prime}\in \mathcal{Q}_m; Q\neq Q^{\prime}}(Q\cap Q^{\prime}).
\]
So $\mathcal{H}_m$ is a countable union of line segments with length $3^{m}$. \medskip

We suppose the started point $x\in\mathcal{H}_m\cap \tilde{F}_0$, and $(i3^{m},j3^{m})$ be the center of $D_m(x)$. By definition, $\mathcal{H}_m$ divides $D_m(x)$ into 4 squares with side length $3^{-m}$. We label the four pieces $S_1,...,S_4$ in counterclockwise order, starting with 
\[
S_1=([i3^{m},(i+1)3^{m}]\times [j3^{m},(j+1)3^{m}])\cap \tF_0,
\] 
which is the upper-right piece of $D_m(x)\cap \tilde{F}_0$ (see Figure \ref{fig:Dn}). By rotation and reflection, we can assume $x\in S_1\cap S_4$ and $S_1\neq \emptyset$. We also set 
\begin{align*}
&\partial_o S_1=\partial S_1\cap \mathcal{H}_m,\\
& \partial_o(S_1\cup S_4)=(\partial(S_1\cup S_4))\cap \mathcal{H}_m. 
\end{align*}

\begin{figure}
\centering
\begin{tikzpicture}
\node[anchor=south west,inner sep=0](image) at (0,0) {\includegraphics[width=0.3\textwidth]{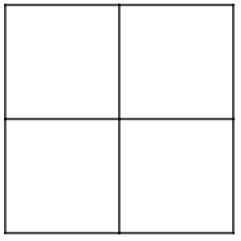}};
\node[align=left] at (3.4,1.2) {$S_4$};
\node[align=left] at (3.4,3.4) {$S_1$};
\node[align=left] at (1.2,3.4) {$S_2$};
\node[align=left] at (1.2,1.2) {$S_3$};
\end{tikzpicture}
\caption{$D_m(x)$ for $x\in \tilde{F}_0$.}
\label{fig:Dn}
\end{figure}
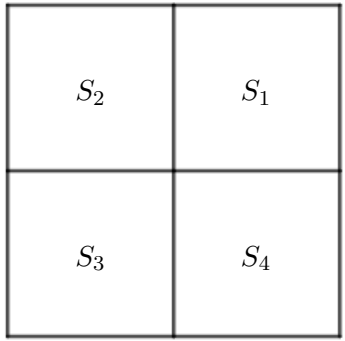

We say $S_i$ is \textit{filled} if $ int(S_i)=\emptyset$ and \textit{unfilled} otherwise. So $W$ can only move in unfilled squares. Since by assumption $S_1$ is unfilled, there are now 7 possible situations of filled and unfilled squares in $D_m(x)$:
\begin{flalign*}
\quad&(a).\ S_1\ is\ unfilled,\ S_2,\ S_3,\ S_4\ are\ filled,\\
&(b).\ S_1 ,\ S_2\ are\ unfilled,\ S_3,\  S_4\ are\ filled,\\
&(c).\ S_1 ,\ S_4\ are\ unfilled,\ S_2,\  S_3\ are\ filled,\\
&(d).\ S_1 ,\ S_2,\ S_3\ are\ unfilled,\ S_4\ is\ filled,\\
&(e).\ S_1 ,\ S_2,\ S_4\ are\ unfilled,\ S_3\ is\ filled,\\
&(f).\ S_1 ,\ S_3,\ S_4\ are\ unfilled,\ S_2\ is\ filled,\\
&(g).\ S_i's\ are\ all\ unfilled.&
\end{flalign*}

For simplicity, we take a translation that maps the center of $D_m(x)$ to the origin, and we focus on the hitting distribution of the process after translation, still denoted by $W.$ More precisely,  we consider two probabilities: 
\begin{equation*}
\begin{aligned}
&P^x(W\ \hbox{hits}\ (\{0\}\times [0,3^{m}/2])\ \hbox{before}\ \sigma_1^m(W)),\\
&P^x(W(\sigma_1^m(W))\in \{3^{m}\}\times [0,3^{m}/2]).
\end{aligned}    
\end{equation*}  
 Following the terminology of Barlow and Bass \cite{barlow1989construction}, we call the piece of the path of $W$ in which $W$ runs from $[0,3^{m}/2]\times \{0\}$ to $\{0\}\times [0,3^{m}/2]$ a \textit{corner move}, and $[0,3^{m}/2]\times \{0\}$ to $\{3^{m}\}\times [0,3^{m}/2]$ a \textit{knight move}.\medskip

Now we give the notion of harmonic functions. Let $(\mathcal{X},d,m)$ be a locally compact, separable metric measure space, and $(\mcE,\mcF)$ be a regular Dirichlet form on it. We say $u\in\mcF$ is \textit{harmonic} in an open $U\subsetneq \mathcal{X}$ if 
\[
\mcE(u,g)=0\quad\hbox{ for every }g\in\mcF\cap C_c(U). 
\]
If $u\in\mcF$ is harmonic in $U$ and is a quasi-continuous version, then for any relatively compact open $V\subsetneq U$   
\[
u(x)=E^x[u(X_{\tau_V})]\qquad\hbox{ for q.e. }x\in V.
\]
In particular, in the case that $u$ is bounded, it holds that
\[
u(x)=E^x[u(X_{\tau_U});\tau_U<\infty]. 
\]

We recall that $(\tmcE_0,\tmcF_0)$ is the Dirichlet form on $\tilde{F}_0$. Let $(\tmcE_0|_A,\tmcF_0|_A)$ be the reflected Dirichlet form of $\tmcE_0$ on a bounded Lipschitz domain $A\subsetneq \tilde{F}_0$, i.e. 
\[
\tmcE_0|_{A}(f,g)=\int_A \nabla f\cdot\nabla g d\mu_0\quad \hbox{ for }f,g\in \tmcF_0|_A. 
\]
Suppose $S_4$ is unfilled, $S_2,S_3$ are filled.
Note that by reflection symmetry, for $f,g\in \tmcF_0|_{S_1\cup S_4}$,
\begin{equation}\label{e:3.1}
\begin{split}  
	\tmcE_0|_{S_1\cup S_4}(f,g)&=\tmcE_0|_{S_1}(f|_{S_1},g|_{S_1})+\tmcE_0|_{S_4}(f|_{S_4},g|_{S_4})\\
	&=\tmcE_0|_{S_1}(f|_{S_1},g|_{S_1})+\frac{\mu_0(S_4)}{\mu_0(S_1)}\tmcE_0|_{S_1}(f|_{S_4}\circ \psi,g|_{S_4}\circ \psi),
\end{split} 
\end{equation}
where we denote by $\psi$  the reflection that maps $S_1$ to $S_4$.\medskip

The following lemma considers harmonic functions on $S_1\cup S_4$ with symmetry or skew-symmetric boundary values. 
\begin{lem}\label{lemma3.1}
Let $f\in C_c^\infty(\partial_o S_1-\{y=0\})$. 
	
(1). Let $h_1\in \tmcF_0|_{S_1\cup S_4}$ such that 
\begin{align*}
		&h_1=f\quad \hbox{on}\ \partial_o S_1-\{y=0\},\\
		&h_1\hbox{ is }\tmcE_0|_{S_1}\hbox{-harmonic in }(S_1\setminus\partial_o S_1)\cup\{y=0\},
\end{align*} 
and 
\[
h_1(x)=h_1(\psi(x))\quad\hbox{ for }x\in S_4.
\]   
Then, $h_1$ is $\tmcE_0|_{S_1\cup S_4} $-harmonic in $(S_1\cup S_4)\setminus\partial_o(S_1\cup S_4)$. 

(2). Let $h_2\in \tmcF_0|_{S_1\cup S_4}$ such that 
	\begin{align*}
		&h_2=f\quad \hbox{on}\ \partial_o S_1-\{y=0\},\\
		&h_2=0\qquad \hbox{on}\ \{y=0\},\\
		&h_2\ \hbox{is $\tmcE_0|_{S_1}$- harmonic in } S_1\setminus \partial_o S_1,
	\end{align*}    
and 
\[
h_2(x)=-\frac{\mu_0(S_1)}{\mu_0(S_4)}h_2(\psi(x))\quad\hbox{ for }x\in S_4.
\]
Then, $h_2$ is $\tmcE_0|_{S_1\cup S_4} $-harmonic in $(S_1\cup S_4)\setminus\partial_o(S_1\cup S_4)$.
\end{lem}
\begin{proof} 
(1). Let $g\in  C_c^1((S_1\cup S_4)\setminus\partial_o(S_1\cup S_4))$. Then $g|_{S_1}$ and $g|_{S_4}\circ \psi\in C_c^1((S_1\setminus\partial_o S_1)\cup \{y=0\})$. Since $h_1$ is $\tmcE_0|_{S_1}$-harmonic in $(S_1\setminus\partial_o S_1)\cup\{y=0\}$, 
\[
\tmcE_0|_{S_1}(h_1|_{S_1},g|_{S_1})=\tmcE_0|_{S_1}(h_1|_{S_1},g|_{S_4}\circ \psi)=0.
\] 
So by \eqref{e:3.1}
\begin{align*}
	\tmcE_0|_{S_1\cup S_4}(h_1,g)
&=\tmcE_0|_{S_1}(h_1|_{S_1},g|_{S_1})+\frac{\mu_0(S_4)}{\mu_0(S_1)}\tmcE_0|_{S_1}(h_1|_{S_4}\circ\psi, g|_{S_4}\circ \psi)\\
&=\tmcE_0|_{S_1}(h_1|_{S_1},g|_{S_1})+\frac{\mu_0(S_4)}{\mu_0(S_1)}\tmcE_0|_{S_1}(h_1|_{S_1}, g|_{S_4}\circ \psi)=0.
\end{align*}

(2). Let $g\in C_c^1((S_1\cup S_4)\setminus\partial_o(S_1\cup S_4))$.By the relation $h_2|_{S_4}=-\frac{\mu_0(S_1)}{\mu_0(S_4)}h_2|_{S_1}\circ \psi$, we notice that  
\begin{align*}
\tmcE_0|_{S_1}(h_2|_{S_1},g|_{S_1})+\frac{\mu_0(S_4)}{\mu_0(S_1)}\tmcE_0|_{S_1}(h_2|_{S_4}\circ \psi,g|_{S_1})=\tmcE_0|_{S_1}(h_2|_{S_1},g|_{S_1})-\tmcE_0|_{S_1}(h_2|_{S_1},g|_{S_1})=0,
\end{align*}
and 
\begin{align*}
\frac{\mu_0(S_4)}{\mu_0(S_1)}\tmcE_0|_{S_1}(h_2|_{S_4}\circ \psi,g|_{S_4}\circ\psi-g|_{S_1})=0,
\end{align*}
as $g|_{S_4}\circ \psi-g|_{S_1}=0$ on $\partial_oS_1$ and $h_2|_{S_4}\circ
\psi$ is $\tmcE|_{S_1} $-harmonic in $S_1\setminus \partial_o S_1$. 

Adding the above two equalities, we see that 
\[
\tmcE_0|_{S_1\cup S_4}(h_2,g)
=\tmcE_0|_{S_1}(h_2|_{S_1},g|_{S_1})+\frac{\mu_0(S_4)}{\mu_0(S_1)}\tmcE_0|_{S_1}(h_2|_{S_4}\circ\psi, g|_{S_4}\circ \psi)=0.
\]
\end{proof}

Throughout the following context, we write $X_t$ for the diffusion associated with $\tmcE_0|_{S_1}$. By Lemma \ref{lemma3.1}, we can get 
\begin{lem}\label{lemma3.2}
$E^x[f\big(W(\mathcal{T}_{\partial_o(S_1\cup S_4)})\big)]=\frac{\mu_0(S_1)}{\mu_0(S_1)+\mu_0(S_4)}E^x[f\big(X(\mathcal{T}_{\partial_o S_1-\{y=0\}})\big)]$ for $ x\in S_1\cap S_4$ and $f\in \mathcal{B}_b(\partial_o S_1-\{y=0\}).$
Here $f$ extends to $\partial_o(S_1\cup S_4)$ by zero extension.
\end{lem} 
\begin{proof}
We denote $T_1=\mathcal{T}_{\partial_o(S_1\cup S_4)},\ T_2=\mathcal{T}_{\partial_oS_1-\{y=0\}}$. First let $f\in C_c^{\infty}(\partial_o S_1\setminus \{y=0\})$, and extend $f$ to $(\partial_o (S_1\cup S_4))$ by zero extension. Set $h(x)=E^x[f(W_{T_1})]$. Then $h(x)$ is harmonic in $(S_1\cup S_4)-\partial_o(S_1\cup S_4)$. Note that 
\[
h(x)=\frac{\mu_0(S_1)}{\mu_0(S_1)+\mu_0(S_4)}(h_1(x)+\frac{\mu_0(S_4)}{\mu_0(S_1)}h_2(x)),
\]
 where $h_1,h_2$ are defined with $f$ as in Lemma \ref{lemma3.1}. Since $h_2(x)=0$ as $x\in S_1\cap S_4$, \[
h(x)=\frac{\mu_0(S_1)}{\mu_0(S_1)+\mu_0(S_4)}h_1(x)=\frac{\mu_0(S_1)}{\mu_0(S_1)+\mu_0(S_4)}E^x[f(X(T_2))].
\] 
 This proves the desired equality for smooth $f$. The equality extends to hold for $f\in \mathcal{B}_b(\partial_o S_1-\{y=0\})$ by the monotone class theorem for functions \cite[A.1.3.]{chen2012symmetric}.  
\end{proof}

Lemma \ref{lemma3.2} can be thought as a principle of folding the process in two squares to one square. 

\subsection{Corner moves}
Divide $\partial (S_1\cup S_4)\cap \mathcal{H}_m$ into 12 line segments, each of length $3^{m}/2$, which we label counterclockwise $L_1,...,L_{12}$, starting with $L_1=\{3^{m}\}\times [0,3^{m}/2]$. Define $T=\mathcal{T}_{\partial_o(S_1\cup S_4)}(W)$. Set $p_i(x)=P^x(W_T\in L_i)$.

\begin{thm}\label{theorem3.3}	
$(1)$. If $S_4$ is filled then $p_6(x)\geq \frac{1}{6}$ for $x\in [0,3^{m}/2]\times \{0\}$.

$(2)$. If $S_4$ is unfilled then $p_6(x)\geq \frac{\mu_0(S_1)}{6(\mu_0(S_1)+\mu_0(S_4))}$ for $x\in [0,3^{m}/2]\times \{0\}$.
  \begin{proof}
$(1)$.  Set $S=\inf\{t>0,W_t\in\{3^{m}/2\}\times [0,3^{m}]\}$. Then 
\begin{align*} 
P^x(W_T\in L_1)&=P^x(W_T\in L_1,S<T)\\
&=E^x[P^{W_S}(W_T\in L_1);S<T]\\
&=E^x[P^{W_S}(W_T\in L_6);S<T]\\
&\leq P^x(W_T\in L_6).
\end{align*}
So $p_1\leq p_6$, similarly $p_2\leq p_5,\ p_3\leq p_4$. 

Next, we set $R_0=0$, and 
\begin{align*} 
V_i&=\inf\{t>R_i,W_t\in [0,3^{m}]\times \{3^{m}/2\}\}\quad\hbox{ for }i\geq 0,\\ R_i&=\inf\{t>V_{i-1};\ W_t\in[0,3^{m}]\times\{0\}\}\quad\hbox{ for }i\geq 1. 
\end{align*}  
Then $p_6=\sum_{i=0}^{\infty}P^x(W_T\in L_6,\ R_i<T<R_{i+1})$, and by the strong Markov property,
\begin{align}
    p_5&=\sum_{i=0}^\infty P^x(W_T\in L_5,\ R_i<T<R_{i+1})\nonumber\\ 
    &=\sum_{i=0}^\infty P^x(W_T\in L_5,\ V_i<T<R_{i+1})\nonumber\\
    &=\sum_{i=0}^\infty E^x[P^{W_{V_i}}(W_T\in L_5,\ T<R_1);\ V_i<T]\nonumber\\
    &=\sum_{i=0}^\infty E^x[P^{W_{V_i}}(W_T\in L_6,\ T<R_1);\ V_i<T]\nonumber\\
    &=\sum_{i=0}^\infty P^x(W_T\in L_6;V_i<T<R_{i+1})\leq p_6.\label{e:1}
    \end{align}
 Similarly, $p_4\leq p_5$ by applying the reflection symmetry about the diagonal connecting $(0,3^{m})$ and $(3^{m},0)$. 

To conclude, we proved $p_6\geq p_5\geq p_4$ and $p_4+p_5+p_6\geq p_1+p_2+p_3$, hence 
\[
p_6\geq \frac13\sum_{i=4}^6p_i\geq\frac{1}{6}\sum_{i=1}^6p_i=\frac16.
\]

For $(2)$, set $f=\mathds{1}_{L_6}$ be the indicator function of $L_6$ on $\partial_o (S_1\cup S_4)$ in Lemma \ref{lemma3.2} and use $(1)$ above, 
\[
p_6(x)=P^x(W_T\in L_6)=\frac{\mu_0(S_1)}{\mu_0(S_1)+\mu_0(S_4)}P^x(X(\mathcal{T}_{\partial_o S_1-\{y=0\}})\in L_6)\geq \frac{\mu_0(S_1)}{6(\mu_0(S_1)+\mu_0(S_4))}.
\]
  \end{proof}  
\end{thm}

\subsection{Knight moves} Let $L_1=\{3^{m}\}\times [0,3^{m}/2]$, and $L_i,1\leq i\leq 16$ be counterclockwise labeled on $\partial [-3^{m},3^{m}]^2$, each of length $3^{m}/2$. Let $T=\sigma_1^m(W),\ p_i(x)=P^x(W_T\in L_i)$. Note that if some $S_j$ is filled, then some of the $p_i(x)$ are zero. In the rest of this section, we will prove $p_1(x)\geq c$ for a positive constant $c$ independent of $x$ and $m$. 

First, we prove a special case of knight moves by considering the reflected process in a square. Recall that $X_t$ is the reflected diffusion in $S_1$. 
\begin{prop}\label{prop3.4}
$P^x(X_T\in L_1)\geq \frac14$ for $x\in [0,3^{m}/2]\times\{0\}$, where $T=\mathcal{T}_{L_1\cup L_2\cup L_3\cup L_4}$. 
\end{prop}
\begin{proof}
Using a similar argument as in \eqref{e:1}, we have $p_1\geq p_2$, and $p_1+p_2\geq p_3+p_4$.
So $p_1(x)\geq \frac{1}{4}$.
\end{proof}

For general cases, set 
\begin{align*}
&O=(0,0),\\
&V_i=(3^{m}\cos\frac{i-1}{2}\pi,3^{m}\sin\frac{i-1}{2}\pi),\qquad i=1,...,4,\\
&\Gamma=\partial[-3^{m},3^{m}]^2,\\
&A_i=\overline{OV_i},\qquad i=1,...,4,\\
&\eta_0=0,\ Y_0=1,\\
&\eta_r=\{t>\eta_{r-1},X_t\in \Gamma\cup (\cup_{i=1}^4 A_i)\setminus A_{Y_{r-1}}\}\hbox{, where }Y_r\hbox{ is the random variable such that }W_{\eta_r}\in A_{Y_r},\\
&N=\min\{r:X_{\eta_r}\in \Gamma\}.
\end{align*}

 Due to Lemma \ref{lemma3.2}, $Y_r$ has the transition probability 
\[
P^x(Y_r=j|\,Y_{r-1}=i,r<N)=p_{i,j} :=\begin{cases}
\frac{\mu_0(S_i)}{\mu_0(S_{i})+\mu_0(S_{i-1})}&\hbox{ if } j=i+1,\\
\frac{\mu_0(S_{i-1})}{\mu_0(S_{i})+\mu_0(S_{i-1})}&\hbox{ if }j=i-1,\\
0&\hbox{ otherwise}. 
\end{cases} 
\]
\medskip

Let $\phi_{ i}$ be the rotation taking $A_{ i}$ to $A_1$, more precisely, for $x=(x^{(1)},x^{(2)})$,
$$\phi_{i}(x)=(x^{(1)}\cos\frac{i-1}{2}\pi+x^{(2)}\sin\frac{i-1}{2}\pi,-x^{(1)}\sin\frac{ i-1}{2}\pi+x^{(2)}\cos\frac{ i-1}{2}\pi).$$ 
Set $\Delta$ be an isolated point which is not in $\R^2$. Define the diffusion in $S_1\cup S_4$ killed on $\partial_o (S_1\cup S_4)$ by:
\begin{align*}
 B_t^r&=\phi_{Y_r}(W_{\eta_r+t}),&&\qquad 0\leq t\leq \eta_{r+1}-\eta_r,\qquad 0\leq r\leq N-1,\\
 B_t^r&=\Delta,&&\qquad t>\eta_{r+1}-\eta_r, \qquad 0\leq r\leq N-1.
\end{align*}


\begin{lem}\label{lemma3.5}
	$B_0^r$ is independent with $Y_r$ under $P^x$ and condition $\{r<N\}$.
\end{lem}		
\begin{proof}
Take $A\subset A_1$ and $i\in \{1,2,3,4\}$. It suffices to prove 
\begin{equation}\label{e:3.2} 
    P^x(B_0^r\in A, Y_r=i|r<N)=P^x(B_0^r\in A|r<N)P^x(Y_r=i|r<N).
\end{equation}
We prove \eqref{e:3.2} by induction. 

First set $r=1$. Note that if $S_4$ is filled, then $Y_1=2$ almost surely on $\{1<N\}$, and \eqref{e:3.2} is obvious. If $S_4$ is unfilled, then $Y_1=2,4$ almost surely and 
\begin{equation}\label{e:3.3}
\begin{aligned}
P^x(B_0^1\in A,Y_1=2,1<N )&=P^x\big(W(\mathcal{T}_{\partial_o (S_1\cup S_4)})\in \phi_2^{-1}(A)\big)\\
&=\frac{\mu_0(S_1)}{\mu_0(S_1)+\mu_0(S_4)}P^x\big(X(\mathcal{T}_{\partial_o S_1-\{y=0\}})\in \phi_2^{-1}(A)\big)\\
&=P^x(Y_1=2|1<N)P^x\big(X(\mathcal{T}_{\partial_o S_1-\{y=0\}})\in \phi_2^{-1}(A)\big)
\end{aligned}
\end{equation}
where the second equality holds by Lemma \ref{lemma3.2}, and recall that $X_t$ is the diffusion in $S_1$ associated with $\tmcE_0|_{S_1}$. The same holds for $Y_1=4$:
\begin{equation}\label{e:3.4}
 P^x(B_0^1\in A,Y_1=4,1<N )=P^x(Y_1=4|1<N)P^x(X(\mathcal{T}_{\partial_o S_1-\{y=0\}})\in \phi_2^{-1}(A)).   
\end{equation}
The equalities \eqref{e:3.3} and \eqref{e:3.4} together give 
\begin{equation}\label{e:3.5}
P^x(B_0^1\in A,1<N )=P^x(X(\mathcal{T}_{\partial_o S_1-\{y=0\}})\in \phi_2^{-1}(A)).
\end{equation}
Combining \eqref{e:3.3}, \eqref{e:3.4} and \eqref{e:3.5}, we get \eqref{e:3.2}. 

Next, we prove the claim for $r\geq 2$ by induction. Suppose the lemma is valid for $r-1$. 
Then, 
\begin{equation}\label{e:3.6}
\begin{aligned}
   &\quad\ P^x(B_0^r\in A,Y_r=i,Y_{r-1}=i-1,r<N)\\
    &=E^x[P^{B_0^{r-1}}(B_{\eta_1}^0\in \phi_2^{-1}(A));Y_{r-1}=i-1,r-1<N]\\
    &=E^x[p_{i-1,i}P^{B_0^{r-1}}(B_{0}^1\in A);Y_{r-1}=i-1,r-1<N] \\
 &=p_{i-1,i}E^x[P^{B_0^{r-1}}(B_{0}^1\in A);r-1<N]P^x(Y_{r-1}=i-1|r-1<N)\\
 &=p_{i-1,i}P^x(B_0^{r}\in A,r<N)P^x(Y_{r-1}=i-1|r-1<N).
 \end{aligned}
\end{equation}
Take $A=A_1$, then 
\[
P^x(Y_r=i,Y_{r-1}=i-1,r<N)=p_{i-1,i}P^x(r<N)P^x(Y_{r-1}=i-1|r-1<N),
\]
which implies 
\begin{equation}\label{e:3.7}
p_{i-1,i}P^x(Y_{r-1}=i-1|r-1<N)=P^x(Y_r=i,Y_{r-1}=i-1|r<N). 
\end{equation}
The equalities \eqref{e:3.6} and \eqref{e:3.7} together imply
\begin{equation}\label{e:3.8}
\begin{aligned}
 &\quad\ P^x(B_0^r\in A,Y_r=i,Y_{r-1}=i-1,r<N)\\
 &=P^x(B_0^{r}\in A,r<N)P^x(Y_r=i,Y_{r-1}=i-1|r<N). 
\end{aligned}
\end{equation}
By a same argument, 
\begin{equation}\label{e:3.9}
\begin{aligned}
&\quad\ P^x(B_0^r\in A,Y_r=i,Y_{r-1}=i+1,r<N)\\
&=P^x(B_0^{r}\in A,r<N)P^x(Y_{r}=i,Y_{r-1}=i+1|r<N).
\end{aligned}
\end{equation}
Combining \eqref{e:3.8} and \eqref{e:3.9}, we have  
\begin{equation}\label{e:3.10}
  P^x(B_0^r\in A,Y_r=i,r<N)=P^x(B_0^{r}\in A,r<N)P^x(Y_{r}=i|r<N),  
\end{equation}
which immediately gives \eqref{e:3.2}. 
\end{proof}

\noindent\textbf{Remark.} By \eqref{e:3.5}, 
\begin{equation}\label{e:3.11}
\begin{split}
P^x(B_0^r\in A,r<N)&=E^x[P^{B_0^{r-1}}(B_0^1\in A,1<N);r-1<N]
\\&=E^x[P^{B_0^{r-1}}(X(\mathcal{T}_{\partial_o S_1-\{y=0\}})\in \phi_2^{-1}(A));r-1<N].  
\end{split}
\end{equation}\medskip

Let $\omega$ be a markov kernel on $[0,3^{m}]\times \{0\}$:\[
\omega(x,A)=P^x(X(\mathcal{T}_{\partial_o S_1-\{y=0\}})\in \phi_2^{-1}(A))
\] 
for $x\in [0,3^{m}]\times \{0\}, A\subset [0,3^{m}]\times \{0\}$. \medskip

\begin{lem}\label{lemma3.6} 
For the reflected process $X_t$ starting with $x\in [0,3^{m}]\times\{0\}$, define $\bar{\eta}_r,\bar{Y}_r$ and $\bar{N}$ for $X_t$ in a same way as $\eta_r,Y_r$ and $N$ as for $W_t$. More precisely, let $\bar{\eta}_0=0,\ \bar{Y}_0=1$, and for $r\geq 1$,
\begin{align*}
&\bar{Y}_r=\begin{cases}
1\quad\hbox{ if }r\hbox{ is even},\\
2\quad\hbox{ if }r\hbox{ is odd},
\end{cases}\\
&\bar\eta_r=\inf\{t>\bar\eta_{r-1}: X_t\in \partial_oS_1\setminus A_{Y_{r-1}}\},\\
&\bar{N}=\min\{r>0:\,X_{\bar{\eta}_r}\in \Gamma\}. 
\end{align*}
Then 
\[
P^x(B_0^r\in A ,r<N)=P^x(X_{\bar{\eta}_r}\in \phi_{\bar{Y}_r}^{-1}(A),r<\bar{N}). 
\] 
\end{lem}
\begin{proof}
It suffices to prove 
\[
P^x(B_0^r\in A ,r<N)=\int\cdots\int \omega(x,d y_1)\omega(y_1,d y_2)...\omega(y_{r-1},A)\quad\hbox{ for every }A\subset A_1. 
\]
We prove this by induction. For $r=1$, the claim follows from  \eqref{e:3.11} and the definition of $\omega(x,A)$. Next, for $r\geq 2$, suppose the lemma is valid for $r-1$, then for every $A\subset A_1$, by \eqref{e:3.11} and the induction, 
\begin{align*}
P^x(B_0^r\in A ,r<N)&=E^x[P^{B_0^{r-1}}(X(\mathcal{T}_{\partial_o S_1-\{y=0\}})\in \phi_2^{-1}(A));r-1<N]\\
&=\int \cdots\int P^{y_{r-1}}(X(\mathcal{T}_{\partial_o S_1-\{y=0\}})\in \phi_2^{-1}(A))\omega(x,d y_1)...\omega(y_{r-2},dy_{r-1})\\
&=\int\cdots\int \omega(x,d y_1)\omega(y_1,d y_2)...\omega(y_{r-1},A). 
\end{align*}
\end{proof}

\begin{thm}\label{theorem3.7}
There exists a constant $c>0$ independent of $x,m$, such that $$p_1(x)\geq c\qquad \hbox{ for } x\in [0,3^{m}/2]\times \{0\}.$$
\end{thm}
\begin{proof}
Let 
\begin{align*}
q_r(i)=E^x[P^{B_{0}^{r-1}}(X(\mathcal{T}_{\partial_o S_1-\{y=0\}})\in L_i);r-1<N] \hbox{ for } i=1,4.
\end{align*}
Note that by Lemma \ref{lemma3.6} and the reflection symmetry, for $r\geq 0$,  
\begin{align*}
q_{2r+1}(1)&=E^x[P^{X_{\bar{\eta}_{2r}}}(X(\mathcal{T}_{\partial_o S_1-\{y=0\}})\in L_1);2r<\bar{N}]\\
&=P^x(X_{\bar{\eta}_{2r+1}}\in L_1,2r<\bar{N}),\\
q_{2r+2}(4)&=E^x[P^{\phi_2(X_{\bar{\eta}_{2r+1}})}(X(\mathcal{T}_{\partial_o S_1-\{y=0\}})\in L_4);2r+1<\bar{N}]\\
&=E^x[P^{X_{\bar{\eta}_{2r+1}}}(X(\mathcal{T}_{\partial_o S_1-\{y=0\}})\in L_1);2r+1<\bar{N}]\\
&=P^x(X_{\bar{\eta}_{2r+2}}\in L_1,2r+1<\bar{N}).
\end{align*}
Then, by Proposition \ref{prop3.4},  it holds that
\begin{equation}\label{e:3.12}
\sum_{r=0}^\infty q_{2r+1}(1)+\sum_{r=0}^\infty q_{2r+2}(4)=P^x(X_T\in L_1)\geq\frac14. 
\end{equation}

Next, we observe
\begin{align*}
&\quad\ P^x(W_{\eta_r}\in L_1,r-1<N,Y_{r-1}=2)\\
&=E^x[P^{W_{\eta_{r-1}}}(W_{\eta_1}\in L_1);r-1<N, Y_{r-1}=2]\\
&=E^x[P^{B_0^{r-1}}(B^{0}_{\mathcal{T}_{\partial_o(S_1\cup S_4)}}\in L_{13});r-1<N,Y_{r-1}=2]\\
&=E^x\Big[\frac{\mu_0(S_1)}{\mu_0(S_1)+\mu_0(S_2)}P^{B_0^{r-1}}(X(\mathcal{T}_{\partial_o S_1-\{y=0\}}\in L_4);r-1<N,Y_{r-1}=2\Big]\\
&=\frac{\mu_0(S_1)}{\mu_0(S_1)+\mu_0(S_2)}q_r(4)P(Y_{r-1}=2|r-1<N),
\end{align*}
where the third equality uses Lemma \ref{lemma3.2}, and the last equality uses Lemma \ref{lemma3.5}.
Similarly, 
\[
P^x(W_{\eta_r}\in L_1,r-1<N,Y_{r-1}=1)=\frac{\mu_0(S_1)}{\mu_0(S_1)+\mu_0(S_4)}q_r(1)P(Y_{r-1}=1|r-1<N).
\]
Summing over $r$, and noticing that $Y_{2r+1}\in \{2,4\},Y_{2r}\in \{1,3\}$ almost surely, we get 
\begin{equation*}
\begin{split} 
p_1(x)&=\sum_{r=0}^{\infty}P^x(W_{\eta_{2r+1}}\in L_1,2r<N,Y_{2r}=1)+\sum_{r=0}^{\infty}P^x(W_{\eta_{2r+2}}\in L_1,2r+1<N,Y_{2r+1}=2)\\
&=\frac{\mu_0(S_1)}{\mu_0(S_1)+\mu_0(S_4)}\sum_{r=0}^{\infty}q_{2r+1}(1)P(Y_{2r}=1|2r<N)\\
&\quad\ +\frac{\mu_0(S_1)}{\mu_0(S_1)+\mu_0(S_2)}\sum_{r=0}^{\infty}q_{2r+2}(4)P(Y_{2r+1}=2|2r+1<N)\\
&\geq C_1\big(\sum_{r=0}^\infty q_{2r+1}(1)+\sum_{r=0}^\infty q_{2r+2}(4)\big)\\
&\geq C_1/4, 
\end{split} 
\end{equation*}
for some $C_1>0$, where in the first inequality we use the fact that $\frac{\mu_0(S_1)}{\mu_0(S_4)},\frac{\mu_0(S_1)}{\mu_0(S_2)}, P(Y_{2r}=1|2r<N), P(Y_{2r+1}=2|2r+1<N)$ are uniformly bounded below by a positive constant, and the second inequality follows from \eqref{e:3.12}.
\end{proof}

\section{Harnack Inequality}\label{section4}
For $m\in\Z$ and $x\in\tilde{F}_0$, we recall that $D_m(x)$ is a block of $2\times 2$ squares, each of side length $3^{m}$, around $x$. We denote the center of $D_m(x)$ by $x_0=(x_0^{(1)},x_0^{(2)})$ and write 
\[
\partial_oD_m(x)=\partial D_m(x)\cap \mathcal{H}_m.
\] 
Let 
\[
G_m(x)=D_m(x)\cap \{z:\|z-x_0\|_{\infty}\leq 
\frac{2}{3}\cdot3^{m}\},
\] 
where $\|y\|_{\infty}=\sup \{|y^{(1)}|,|y^{(2)}|\},\ y\in \R^2$. 

In this section, We prove the following \textit{elliptic Harnack inequality} of $W_t$.

\begin{thm}\label{theorem4.1}
There is a positive constant $\theta$ so that 
\[
f(y)\leq \theta\,f(z)\qquad \hbox{ for  }y,z\in G_m(x)
\]
for every $x\in \tilde{F}_0$, $m\in\Z$, and $f\in\tmcF_0$ that is non-negative and harmonic in $D_m(x)\setminus\partial_o D_m(x)$. 
\end{thm}
	

\begin{figure}
	\centering
	\begin{tikzpicture}
	\node[anchor=south west,inner sep=0](image) at (0,0) {\includegraphics[width=0.4\textwidth]{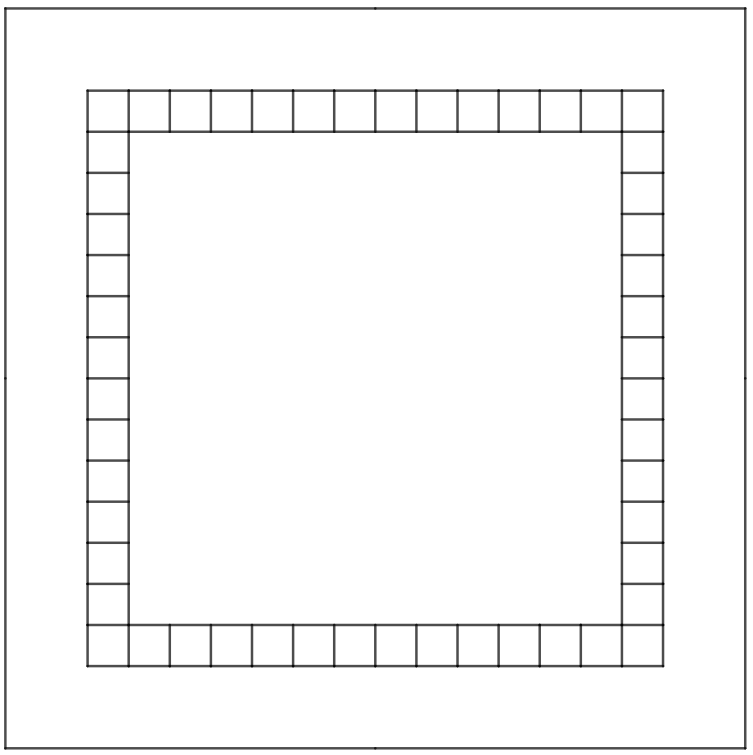}};
    \node[align=left] at (3,3) {$G_m(x)$};
    \node[align=left] at (5.4,0.44) {$D_m(x)$};
	\end{tikzpicture}
    \caption{$D_m(x)$ and $G_m(x)$.}
	\label{fig:chain of squares}
\end{figure}

Following \cite{barlow1989construction}, we use the knight move and corner move constructions to prove the elliptic Harnack inequality. Recall that $W_t$ is the diffusion process on $\tilde{F}_0$, and we set 
\[
\tau=\mathcal{T}_{\partial_o D_m(x)}. 
\]

\begin{lem}\label{lemma4.2}
There exists $\delta>0$ which is independent of $D_m(x)$ such that if $y,z\in G_m(x)$ and $\gamma (t),\ 0\leq t\leq 1$ is a continuous curve in $D_m(x)$ from $y$ to $\partial_o D_m(x)$, then 
\begin{equation}\label{e:4.1}
    P^z(W\ hits \ \gamma \   before \ \tau) >\delta.
\end{equation}
\begin{proof}
Set 
\[
\partial_oG_m(x)= \partial([x_0^{(1)}-\frac{2}{3} 3^{m},x_0^{(1)}+\frac{2}{3} 3^{m}]\times [x_0^{(2)}-\frac{2}{3} 3^{m},x_0^{(2)}+\frac{2}{3} 3^{m}])\cap \tilde{F}_0.
\]
By applying the corner move (Theorem \ref{theorem3.3}) and knight move (Theorem \ref{theorem3.7}) construction inside the squares each of side length $3^{-m-2}$, surrounding $G_m(x)$ (see Figure \ref{fig:chain of squares}) for the process starting at $W_{\mathcal{T}_{\partial_oG_m(x)}}$, and by a same idea of \cite{barlow1989construction}, there is $\delta>0$ depending only on the number of small squares such that  
\[
P^z(W\hbox{ hits every }L\in \mathcal{H}_{m+2}\hbox{ that satisfies }L\cap \partial_oG_m(x)\neq\emptyset)>\delta.
\]
Noticing that in the above event $W$ divides $D_m(x)$ into two parts, one of which contains $G_m(x)$, it holds that
\[
P^z(W\ \hbox{hits} \ \gamma \   \hbox{before} \ \tau) \geq P^z(W\hbox{ hits every }L\in \mathcal{H}_{m+2}\hbox{ that satisfies }L\cap \partial_oG_m\neq\emptyset)>\delta.
\]
\end{proof}
\end{lem}

\begin{proof}[Proof of Theorem \ref{theorem4.1}]
For every Borel subset $A\subset \partial_oD_m(x)$, set
\[
h_m(y,A)=P^y(W_{\tau}\in A)\quad\hbox{ for }y\in D_m(x)\setminus\partial_oD_m(x). 
\]
Note that $h_m(y,\bullet)$ is the hitting distribution of $W$ on $\partial_o D_m(x)$ starting from $y$.

Let $M_t=h_m(W_{t\wedge \tau},A)$, then $M_\cdot$ is a $P^y$-martingale. Obviously $0\leq M\leq 1$, and $M$ is continuous. Let $0<\eta<1$ and let 
\[
T=\inf\{t\geq 0:M_t<\eta h_m(y,A)\}\wedge \tau.
\] 
Then 
\begin{align*}
h_m(y,A)&=E^y[h_m(W_T,A)]\\
&=E^y[h_m(W_{\tau},A);T=\tau]+E^y[h_m(W_T,A);T<\tau]\\
&\leq P^y(T=\tau)+\eta h_m(y,A)P^y(T<\tau),
\end{align*}
which implies $P^y(T<\tau)\leq \frac{1-h_m(y,A)}{1-\eta h_m(y,A)}<1$ and $P^y(T=\tau)>0$. Hence, there is a curve $\gamma$ that satisfies
\[
\gamma(0)=y,\ \gamma(1)\in \partial_o D_m(x)\hbox{ and }h_m(\gamma(t),A)\geq \eta h_m(y,A) \hbox{ for any }0\leq t\leq 1.
\]
Let 
\[
S=\inf\{t\geq 0,W_t\in \gamma\}.
\] 
By Lemma \ref{lemma4.2}, for $y,z$ in $G_m(x)$, $P^z(S<\tau)>\delta$. Then 
\begin{align*}
h_m(z,A)&=E^z[h_m(W_{S\wedge \tau},A)]\\
&\geq E^z[h_m (W_S,A);S<\tau]\\
&\geq \delta \eta h_m(y,A).
\end{align*}
Letting $\eta\rightarrow 1$, we get 
\[
h_m(z,A)\geq \delta h_m(y,A).
\] 
The theorem follows, noticing that 
$f(w)=\int_{\partial_oD_m(x)}h_m(w,du)f(u)$ for $w\in D_m(x)\setminus \partial_oD_m(x)$ and non-negative $f$ that is harmonic in $D_m(x)\setminus \partial_oD_m(x)$.    
\end{proof}

\section{Resistance Estimates}\label{section5}
For $n\in \Z$, set 
$$
F_0^n=[0,3^n]^2\cap\tilde{F}_0,\hbox{ and }\partial_o F_0^n=\partial F_0^n\cap\mathcal{H}_{n}.
$$

Let 
$$
C_1=(0,0),C_2=(3^n,0),C_3=(3^n,3^n),C_4=(0,3^n)
$$
be the four vertices of $F_0^n$, and let 
$$
L_i=\overline{C_iC_{i+1}}\ \hbox{ for }i=1,2,3,4
$$
be the four border segments of $F_0^n$, where we use the cyclic notation $4+1=1$. Note that the definitions of $C_i$ and $L_i$ actually depend on $n$, but we omit the index $n$ for simplicity.

Let $R_n$ be the \textit{effective resistance} between $L_4$ and $L_2$ defined through
\begin{equation}\label{e:5.1}
R_n^{-1}=\inf\{\tmcE{_0}|_{F_0^n}(f):f|_{L_4}=0,\,f|_{L_2}=1,\, f\in C(F_0^n)\cap W^{1,2}(F_0^n)\}. 
\end{equation} 
It is known that the infimum of \eqref{e:5.1}  is attained by a unique function $ u_n$. In this section, we prove the following resistance estimates. 

\begin{thm}\label{Theorem 5.1.} There exist constants $\lambda,C>0$ such that 
\begin{equation}\label{e:5.2}
 \begin{aligned}
&C^{-1}\lambda^{n}\leq R_n\leq C \lambda^n, \qquad  &&n>0;\\
&C^{-1}\leq R_n\leq C,&&n\leq 0.
 \end{aligned}   
\end{equation}
\end{thm}
\noindent\textbf{Remark.} When $n\leq 0$, $F_0^n$ is a square and $\tilde{\mcE}{_0}|_{F_0^n}(f)=\int_{F_0^n}|\nabla f|^2dx$, hence $R_n=1$. So we only need to prove Theorem \ref{Theorem 5.1.} for $n>0$. \medskip

\subsection{Resistance and flows on $F_0^n$}\label{section51}

Before proceeding, we provide an alternative variational characterization of the resistance $R_n$ in terms of flows. 

Firstly, we introduce the notion of $BV$ functions. Let $U$ be a bounded open set in $\R^2$ with Lipschitz boundary. A function $f\in L^1(U)$ is said to be of \textit{bounded variation} on $U$ if 
\begin{align*}
    \sup\{\int_U f\ \mathrm{ div}\varphi\ dx|\,\varphi\in { C^1_c(U;\R^2)},|\varphi|\leq1\}<\infty.
\end{align*}
We denote by $BV(U)$ the space of functions of bounded variation on $U$. Note that the weak  partial derivatives of a $BV$ function $f$ are Radon measures. 


We now introduce the notion of  flows. Let $\sigma$ denote the $1$-dimensional Hausdorff measure. For a vector field $J=(J_1,J_2)$ so that $J_i\in BV(F_0^n)$, we say $J$ is $matching$ on $F_0^n$, if for any $Q_1,Q_2\in \mathcal{Q}_0(F_0^n)\hbox{ with}\ \sigma(Q_1\cap Q_2)>0$, $J$ satisfies 
\[
J\cdot \nu_1=-J\cdot \nu_2,\qquad \hbox{$\sigma$-a.e. on } Q_1\cap Q_2,
\]
where $\nu_1,\nu_2$ are the outer normals to $\partial Q_1$ and $\partial Q_2$ respectively. We call a vector field $J=(J_1,J_2)$ on $F_0^n$ a \textit{flow between $L_4$ and $L_2$} if
\begin{equation}\label{e:5.6}
\begin{aligned}
    &(a).\  J_i\in BV(F_0^n),\qquad i=1,2,\\
    &(b).\ \mathrm{div} J=0\ \hbox{in}\ Q,\qquad \hbox{for each } Q\in \mathcal{Q}_0(F_0^n),\qquad\qquad\qquad\qquad\qquad\qquad\qquad\qquad\\
    &(c).\ J\cdot \nu=0\ \hbox{on}\ \partial F_0^n-(L_2\cup L_4),\\
    &(d).\ J\hbox{ is matching on } F_0^n.
\end{aligned}
\end{equation}

\begin{prop}\label{proposition 5.3}
Let $J$ be a flow between $L_4$ and $L_2$, and $u$ be a feasible function for \eqref{e:5.1}. Then 
\[
\int_{F_0^n}\nabla u\cdot Jdx=\int_{L_2}J\cdot \nu d\sigma.
\]
\end{prop}
\begin{proof}
By the Gauss-Green formula \cite[Section 9.6.5, p506]{maz2013sobolev}, for $Q\in \mathcal{Q}_0(F_0^n)$,
\[
\int_{Q}\nabla u\cdot J dx+\int_{Q}u\ \mathrm{div} Jdx=\int_{\partial Q}u(J\cdot \nu) d\sigma. 
\]
Summing over $Q$ in $\mathcal{Q}_0(F_0^n)$, we have
\begin{align*}
\int_{F_0^n}\nabla u\cdot J dx&=\sum_{Q\in\mathcal{Q}_0(F_0^n)} \int_{Q}\nabla u\cdot J dx\\
&=\sum_{Q\in \mathcal{Q}_0(F_0^n)}\int_{\partial Q}u(J\cdot \nu) d\sigma{=\int_{L_2\cup L_4}u(J\cdot \nu)d\sigma}=\int_{L_2}J\cdot \nu d\sigma.
\end{align*}
\end{proof}
\medskip

For vector value functions $I,J$ on $F_0^n$, set 
\[
{ E_n}(I,J)=\sum_{Q\in \mathcal{Q}_0(F_0^n)}(\mu_0(Q))^{-1}\int_{Q}I\cdot J\  dx,
\]
and write $E_n(I)=E_n(I,I)$ for short.

\begin{thm}\label{theorem5.4}
Let $u_n$ be the unique function that minimizes \eqref{e:5.1}, and let 
\[
I_n=R_n\sum_{Q\in \mathcal{Q}_0(F_0^n)}\mu_0(Q)  \mathds{1}_Q\nabla u_n.
\]
Then,  
\begin{equation}\label{e:5.7}
R_n=\inf\{E_n(J)|J \hbox{ is a flow between $L_4$ and $L_2$ with } \int_{L_2}J\cdot \nu d\sigma=1\},  
\end{equation}
with $I_n$ being the unique flow minimizing \eqref{e:5.7}.
\end{thm}
Before proving this theorem, we first state a lemma concerning $u_n$. For a bounded open set $U$ in $\R^2$ with Lipschitz boundary and $f\in W^{1,1}(U)$, we denote by $\Tr f$ the trace of $f$ on $\partial U$.\medskip

\begin{lem}
Let $u_n$ be the function attaining \eqref{e:5.1}. Let $Q_1,Q_2\in \mathcal{Q}_0(F_0^n)$ with $\sigma(Q_1\cap Q_2)>0,$ and denote by $\nu_i$ the outer normal to $\partial Q_i$ for $i=1,2$. Then 
\begin{equation}\label{e:5.3}
\frac{\Tr (\nabla u_n|_{Q_1})\cdot \nu_1}{\Tr(\nabla u_n|_{Q_2})\cdot \nu_2}=-\frac{\mu_0(Q_2)}{\mu_0(Q_1)},\qquad \sigma\hbox{-a.e. on }\ Q_1\cap Q_2.
\end{equation}
\end{lem}
\begin{proof}Since $u_n$ is harmonic in $F_0^n$, 
\[
\tmcE_0(u_n,f)=0,\qquad \hbox{for all}\ f\in C_{c}^{\infty}(F_0^n).
\]  
It suffices to prove for any open interval $A\subset Q_1\cap Q_2$, 
\begin{equation}\label{e:5.4}
   \int_{A}(\mu_0(Q_1)\Tr (\nabla u_n|_{Q_1})\cdot \nu_1+\mu_0(Q_2)\Tr (\nabla u_n|_{Q_2})\cdot \nu_2)\ d\sigma=0. 
\end{equation}
Now take $f_r$ be the cut-off function such that $f_r=1$ on $\overline{A}$, $0\leq f_r\leq1$, and $f_r$ is supported in $A_r$, the $r$-neighborhood of $A$ in $Q_1\cup Q_2$. By Gauss-Green formula, 
\begin{align*}
0=\tmcE{_0}|_{F_0^n}(u_n,f_r)&=\mu_0(Q_1)\int_{Q_1}\nabla u_n\cdot\nabla f_rdx+\mu_0(Q_2)\int_{Q_2}\nabla u_n\cdot\nabla f_rdx\\
&=\mu_0(Q_1)\int_{\partial Q_1}f_r\,\Tr (\nabla u_n|_{Q_1})\cdot \nu_1\,d\sigma+\mu_0(Q_2)\int_{\partial Q_2}f_r\,\Tr (\nabla u_n|_{Q_2})\cdot \nu_2\,d\sigma.
\end{align*}
Letting $r\to 0$, the above equality tends to \eqref{e:5.4}.
\end{proof}
\medskip

\begin{proof}[Proof of Theorem \ref{theorem5.4}]
One can easily check that $I_n$ satisfies (a)-(c) in \eqref{e:5.6}. By \eqref{e:5.3}, for $Q_1,Q_2\in \mathcal{Q}_0(F_0^n)\hbox{ with }\sigma(Q_1\cap Q_2)>0$, $\Tr(I_n|_{Q_1})\cdot \nu_1=-\Tr(I_n|_{Q_2})\cdot \nu_2$, so $I_n$ is matching on $F_0^n$. Thus $I_n$ is a flow between $L_4$ and $L_2$. By Proposition \ref{proposition 5.3}, 
\[
\int_{L_2}I_n\cdot \nu d\sigma=\int_{F_0^n}\nabla u_n\cdot I_n dx
=R_n\sum_{Q\in \mathcal{Q}_0(F_0^n)}\int_{Q}\mu_0(Q)  \nabla u_n\cdot \nabla u_ndx=R_n\tmcE{_0}|_{F_0^n}(u_n)=1.
\]
So $I_n$ is feasible for the right of \eqref{e:5.7}. On the other hand, 
\begin{equation*}
    \begin{aligned}
E_n(I_n)&=R_n^2\sum_{Q\in \mathcal{Q}_0(F_0^n)}(\mu_0(Q))^{-1}\int_{Q}|\mu_0(Q)\nabla u_n|^2 dx.\\
&=R_n^2\tmcE_0|_{F_0^n}(u_n)=R_n.
    \end{aligned}
\end{equation*}

Finally, Let $J$ be a flow between $L_4$ and $L_2$ with $\int_{L_2}J\cdot \nu d\sigma=1$, then by Proposition \ref{proposition 5.3}, 
\[
E_n(I_n,J)=R_n\int_{F_0^n}\nabla u_n\cdot Jdx=R_n\int_{L_2}J\cdot \nu d\sigma=R_n.
\]
So $E_n(I_n,J)=E_n(I_n)$, hence $$E_n(J)-E_n(I_n)=E_n(I_n-J)\geq 0.$$
\end{proof}
\medskip

\subsection{Upper bound estimate of $R_n$}\label{section52}
In this subsection, we estimate the upper bound of $R_n$ by constructing a flow on $F_0^n$. 

Let $\overline{C}=(\frac{1}{2}3^n,\frac{1}{2}3^n)$, and let $T_i$ be the intersection of $F_0^n$ with the triangle with vertices $C_i,C_{i+1},\overline{C}$. Denote by $I_n$ the optimum flow for \eqref{e:5.7}, and write $I_n^{42}=I_n$ to emphasize that the flow is from $L_4$ to $L_2$. Define $I_n^{13},I_n^{24},I_n^{31}$ by counterclockwise rotating  $I_n^{42}$ by $\frac{\pi}{2}, \pi,\frac{3\pi}{2}$  respectively. Define 
\[
U^i=I_n^{(i+2)i}|_{T_i},\quad V^i=I^{(i+1)(i-1)}_n|_{T_i},
\]
where we use the cyclic notations $1-1=4$, $4+1=1$, $3+2=1$ and $4+2=2$ for simplicity. 
Clearly, 
\[
R_n=2E_n|_{T_i}(U^i)+2E_n|_{T_i}(V^i)\qquad \hbox{ for all }i.
\] 
Now, define 
\[
I_n^{41}=U^1-U^4+V^2+V^3,
\] 
and similarly  define the remaining $I_n^{i(i+1)}$, $I_n^{(i-1)i}$ by rotations and reflections.

Let $C_i^{\prime}=\frac{1}{2}(C_i+C_{i+1})$. Let $H$ be a real value function on $\{C_i^{\prime}\overline{C}:\,i=1,2,3,4\}$, and denote $H_i=H(C_i^{\prime}\overline{C})$ with $\sum_i H_i=0$. Let 
\[
E_+(H)=\frac{1}{2}\sum_i H_i^2
\]
be the energy dissipation of $H$ in the crosswire. Write $h=\frac{1}{2}\sum_i|H_i|$. Let 
\begin{equation}\label{e:5.8}
J=\sum_i\sum_{j\neq i}h^{-1}H_i^{+}H_j^{-}I_n^{ij},
\end{equation}
where $H_i^{+},H_i^{-}$ denote the positive and negative parts of $H_i$.

\begin{lem}\label{lemma5.6}
(1). The total flux of $J$ through edge $L_i$ is $H_i,\ i=1,2,3,4$ i.e. 
\[
\int_{L_i}J\cdot \nu d\sigma=H_i. 
\]
(2). $E_n(J)\leq R_n E_+(H)$.
\begin{proof}
The proof is same as that of Proposition 3.2 in \cite{barlow1990resistance}, basing on a direct calculation, noticing that $E_n|_{T_i}(U_i,V_i)=0$ for all $i$ and $\sum_{i=1}^4H_i^2$.
\end{proof}
\end{lem}

Next, we introduce the discrete energies on certain graphs. Let $Q$ be a square in $\mathcal{Q}_{m}$ for some $0\leq m\leq n$, we write $\Phi_{n,Q}$ to be the orientation preserved linear map such that $\Phi_{n,Q}$ maps $[0,3^n]^2$ to $Q$. 
Set \begin{align*}
&X_{0}^{(n)}=\{(\frac{1}{2}3^{n},\frac{1}{2}3^{n})\},\ Y_{0}^{(n)}=\{(\frac{1}{2}3^{n},0),(3^{n},\frac{1}{2}3^{n}),(\frac{1}{2}3^{n},3^{n}),(0,\frac{1}{2}3^{n})\},\\
&Z_{0}^{(n)}=\{(0,0),(0,3^{n}),(3^{n},3^{n}),(3^{n},0)\},\\
&X_n=\bigcup_{Q\in \mathcal{Q}_0(F_0^n)}   \Phi_{n,Q}(X_0^{(n)}),\quad Y_n=\bigcup_{Q\in \mathcal{Q}_0(F_0^n)}\Phi_{n,Q}(Y_0^{(n)}),\\
&Z_n=\bigcup_{Q\in \mathcal{Q}_0(F_0^n)}\Phi_{n,Q}(Z_0^{(n)}),\\
&G_n=X_n\cup Y_n,\ D_n=X_n\cup Z_n,
\end{align*}
and denote
\begin{enumerate}[a)]
	\item[(a).] $x\sim_G y, \hbox{ if } x\in X_n,\ y\in Y_n \hbox{ and } |x-y|={\frac{1}{2}}$;  
	
	\item[(b).] $x\sim_D z, \hbox{ if } x\in X_n,z\in Z_n \hbox{ and } |x-z|={\frac{1}{\sqrt{2}}}$.
\end{enumerate}
Define a weight function $g_n$ on $G_n\times G_n$ by
\[
g_n(x,y)=
\begin{cases}
 \mu_0(Q), &\hbox{ if } x\sim_G y, \hbox{ $x,y\in Q$ for some $Q\in \mathcal{Q}_0(F_0^n)$}.\\
 0, &\hbox{ otherwise},
\end{cases}
\]   
and a weight function $d_n$ on $D_n\times D_n$ by 
\[
d_n(x,y)=
\begin{cases}
 \mu_0(Q), &\hbox{ if } x\sim_D y,\hbox{ $x,y\in Q$ for some $Q\in \mathcal{Q}_0(F_0^n)$}.\\
 0, &\hbox{ otherwise}.
\end{cases}
\]   

Let $(\mathbb{V},\sim, g)$ be a finite weighted graph, where $g$ is the weight function on $\V\times\V$. For a function $f$ on $\V$, define its \textit{(discrete) energy} on $(\V,\sim,g)$ by 
\[
\mathcal{E}_{\V}(f)=\frac12\sum_{x,y\in\V}g(x,y)(f(x)-f(y))^2.
\] 
Note that 

\begin{equation*}
\begin{aligned}
&\mathcal{E}_{G_n}(f)={\frac12}\sum_{x,y\in G_n}g_n(x,y)(f(x)-f(y))^2,\\    
&\mathcal{E}_{D_n}(f)={\frac12}\sum_{x,y\in D_n}d_n(x,y)(f(x)-f(y))^2
\end{aligned}    
\end{equation*}
for functions $f$ on $G_n$ or $D_n$.

For nonempty $B_0,B_1\subset \V$ with $B_0\cap B_1=\emptyset$, we define the \textit{(discrete) effective resistance} between $B_0$ and $B_1$ through
\begin{equation}\label{5}
(R_\V(B_0,B_1))^{-1}=\inf\{\mathcal{E}_{\V}(f)|f:\V\rightarrow\R,\ f|_{B_i}=i,\ i=0,1\}.
\end{equation}

A map $J:\V\times \V\rightarrow\R$ is called a \textit{discrete flow}, if it satisfies:
\begin{align*}
    &{ J(x,y)}=0\qquad\quad\quad\ \hbox{ if } x\nsim y,\\
    &J(x,y)=-J(y,x)\quad\ \hbox{for all }x,y.
\end{align*}
By \cite[Theorem 2.31 p53]{barlow2017random}, we have
\begin{equation}\label{6}
R_{\V}(B_0,B_1)=\inf\{E_\V(J)|J\in \mathscr{I}(B_0,B_1)\}. 
\end{equation}
Here $E_\V(J)=\frac12\sum\limits_{x,y\in\V}g(x,y)^{-1}J(x,y)^2$ is the \textit{discrete energy} of the flow $J$, and $\mathscr{I}(B_0,B_1)$ is the collection of flows $J$ satisfying: 
\begin{align*}
   &\mathrm{Div } J(x)=0,\qquad x\in \V-(B_0\cup B_1),\\
   &\hbox{Flux}(J;B_0)=1,\\
   &\hbox{Flux}(J;B_1)=-1,
\end{align*}
where
\begin{align*}
&\mathrm{Div }J(x):=\sum_{ y\in \V:x\sim y}J(x,y),\\
&\hbox{Flux}(J;A):=\sum_{ x\in A}\mathrm{Div }J(x)\ \hbox{ for }x\in A.
\end{align*}
For more details, we refer to \cite{barlow2017random}.
Write 
\begin{align*}
&R_n^G=R_{G_n}(L_4\cap G_n,L_2\cap G_n)\hbox{ on }(G_n,\sim_G,g_n),\\
&R_n^D=R_{D_n}(L_4\cap D_n,L_2\cap D_n)\hbox{ on }(D_n,\sim_D,d_n).
\end{align*}
Let $u_n^G$ ($u_n^D$) and $I_n^G$ ($I_n^D$) be the optimum potential and optimum flow with respect to \eqref{5} and \eqref{6} for $R_n^G$ ($R_n^D$) respectively.\medskip

\begin{lem}\label{lemma5.7} 
	$R_{n+m}\leq R_n R_m^G$ for $n,m\geq0$. 
\end{lem}
\begin{proof}
We prove the lemma by constructing a flow on $F_0^{n+m}$ with $I_m^G$ and $I_n$. 

Let $Q\in {\mathcal{Q}_{n}(F_0^{n+m})}$, then $Q\cap F_0^{n+m}$ is a copy of $F_0^n$. Define $\tilde{I}_m^G:(3^nG_m\times 3^nG_m)\rightarrow\R$ by
\[
\tilde{I}_m^G(3^nx,3^ny)=I_m^G(x,y)\qquad \hbox{ for } (x,y)\in G_m\times G_m\hbox{ such that }x\sim_Gy.
\] 
Denote by $\tilde{I}_m^G|_Q$ the restriction of $\tilde{I}_m^G$ on $Q\cap  (3^nG_m)$ for $Q\in \mathcal{Q}_{n}(F_0^{n+m})$, and set 
$H=\tilde{I}_m^G|_Q\circ \Phi_{n,Q}$.  
Let $H_i=H(C_i^{\prime}\bar{C})$ and $J_Q=J\circ \Phi_{n,Q}^{-1}:Q\rightarrow\R^2$, where $J$ is defined as in \eqref{e:5.8}.
Define 
\[
J_{n+m}=\sum_{Q\in \mathcal{Q}_{n}(F_0^{n+m})}J_Q.
\]
We can easily check $J_{n+m}$ satisfies \eqref{e:5.6} except (d). If $Q_1,Q_2\in \mathcal{Q}_{n}$ so that $(Q_1\cap Q_2)\cap(3^nG_m)\neq\emptyset$, take $y\in (Q_1\cap Q_2)\cap(3^nG_m)$ and $x_i\in (3^n X_m)\cap Q_i$. Since $\mathrm{Div }\tilde{I}_m^G(y)=0$ i.e. $\tilde{I}_m^G(x_1,y)+\tilde{I}_m^G(x_2,y)=0$, $J_{n+m}$ is matching on the border $Q_1\cap Q_2$. So (d) holds, and $J_{n+m}$ is a flow on $F_0^{n+m}$.

Also we have 
\[
\int_{L_2}J_{n+m}\cdot \nu d\sigma=\sum_{Q\in \mathcal{Q}_{n}(F_0^{n+m})\atop Q\cap L_2\neq \emptyset}\int _{Q\cap L_2}J_{n+m}\cdot \nu d\sigma=1,
\]
where the last equality follows from Lemma \ref{lemma5.6}-(1) and $\hbox{Flux}(I_m^G,L_2\cap G_m)=1$. So $J_{n+m}$ is feasible for \eqref{e:5.7}. Hence by Lemma \ref{lemma5.6}-(2),
\begin{equation*}
\begin{aligned}
E_{n+m}(J_{n+m})&=\sum_{Q\in \mathcal{Q}_{n}(F_0^{n+m})}E_{n+m}|_{Q\cap F_0^{n+m}}(J_Q)\\
&\leq\sum_{Q\in \mathcal{Q}_{0}(F_0^{m})}R_n E_{G_m}|_{Q\cap G_m}(I_m^G|_Q)\\
&=R_n E_{G_m}(I_m^G)=R_n R_m^G.    
\end{aligned}
\end{equation*}
The statement then follows from Theorem \ref{theorem5.4}.
\end{proof}

\begin{lem}\label{lemma5.8} 
	There exists a positive constant $c$ such $c^{-1}R_m^G\leq R_m^D\leq R_m^G$ for $m\geq0$.
\end{lem}
\begin{proof}
First, we show the right-hand inequality by constructing a feasible potential $f$ on $G_m$, using the the optimum potential $u_m^D$ for $R_m^D$. Let 
\begin{align*}
&f(x)=u_m^D(x)\qquad\qquad\qquad\quad   \hbox{ for }x\in X_m,\\
&f(y)=\frac{1}{2}(u_m^D(z_1)+u_m^D(z_2))\quad\hbox{ for }y\in Y_m,
\end{align*}
where in the second line, we choose $z_1,z_2\in Z_m$ so that $y=(z_1+z_2)/2$ and $y,z_1,z_2\in Q$ for some $Q\in \mathcal{Q}_0(F_0^m)$. 

Now, for each $Q\in \mathcal{Q}_0(F_0^m)$, we
denote by $a_1,a_2,a_3,a_4$ the values of $u_m^D$ at the four corners of $Q$ and by $\overline{a}={\frac{1}{4}}\sum_{i=1}^4a_i$ the value of $u_m^D$ at the center of $Q$, then, by the definition of $f$,
\begin{align*}
\mcE_{G_m}|_{Q\cap G_m}(f)&=\mu_0(Q)\sum_{i=1}^4(\frac{a_i+a_{i+1}}{2}-\overline{a})^2\\
&\leq \mu_0(Q)\sum_{i=1}^4(a_i-\overline{a})^2=\mathcal{E}_{D_m}|_{Q\cap D_m}(u_m^D). 
\end{align*}
Summing the above inequality over $Q\in \mathcal{Q}_0(F_0^m)$ implies
\[
(R_m^G)^{-1}\leq\mcE_{G_m}(f)\leq \mcE_{D_m}(u_m^D)=(R_m^D)^{-1},
\] 
so that $R_m^D\leq R_m^G$.

Next, we show the left-hand inequality by constructing a feasible potential $h$ on $D_m$, using the optimum potential $u_m^G$ for $R_m^G$. Let $h:D_m\rightarrow \R$ by 
\begin{equation*}
    \begin{aligned}
&h(x)=u_m^G(x),\ &&x\in X_m,\\
&h(z)=\sum_{x\sim_Dz}\frac{d_m(x,z)}{\sum_{x\sim_D z}d_m(x,z)}u_m^G(x),\qquad &&z\in Z_m-(L_2\cup L_4),\\
&h(z)=0,\ &&z\in L_4,\\
&h(z)=1,\ &&z\in L_2.
    \end{aligned}
\end{equation*} 
 
For $Q\in \mathcal{Q}_0(F_0^m)$, we write $x_Q$ for the center of $Q$; for $Q,Q'\in\mathcal{Q}_0(F_0^m)$ such that $\sigma(Q\cap Q')>0$, we write $y_{Q\cap Q'}$ for the middle point of $Q\cap Q'$. Also, denote 
\[
a_Q=u_m^G(x_Q)\ \hbox{ and }\ b_{Q\cap Q'}=u_m^G(y_{Q\cap Q'})
\]
for short. Then, for each $Q\in \mathcal{Q}_0(F_0^m)$ and $z\in Q\cap Z_m$, we claim that 
\begin{equation}\label{e:5.9}
(h(x_Q)-h(z))^2\leq 16\sum_{Q_1,Q_2\in \mathcal{Q}_0(F_0^m)\atop \sigma(Q_1\cap Q_2)>0,z\in Q_1\cap Q_2}(a_{Q_1}-b_{Q_1\cap Q_2})^2.
\end{equation}
The claim is trivial if $z\in L_2\cup L_4$. If $z\notin L_2\cup L_4$, by the definition of $h(z)$, 
\begin{align*}
(h(z)-h(x_Q))^2&=(\sum_{Q'\in\mathcal{Q}_0(F_0^m):\,z\in Q'}\frac{d_m(x_{Q^{\prime}},z)}{\sum_{Q'\in\mathcal{Q}_0(F_0^m):z\in Q'}d_m(x_{Q^{\prime}},z)}(a_{Q'}-a_Q))^2\\
&\leq\sum_{Q'\in\mathcal{Q}_0(F_0^m):\,z\in Q'}\frac{d_m(x_{Q^{\prime}},z)}{\sum_{Q'\in\mathcal{Q}_0(F_0^m):z\in Q'}d_m(x_{Q^{\prime}},z)}(a_{Q'}-a_Q)^2\\
&\leq \sum_{Q'\in\mathcal{Q}_0(F_0^m):\,z\in Q'}(a_{Q'}-a_Q)^2\\
&\leq 16\sum_{Q_1,Q_2\in \mathcal{Q}_0(F_0^m)\atop \sigma(Q_1\cap Q_2)>0,z\in Q_1\cap Q_2}(a_{Q_1}-b_{Q_1\cap Q_2})^2,
\end{align*}
where in the last inequality, we use the observation
\begin{align*}
(a_{Q'}-a_Q)^2&\leq \sum_{Q_1,Q_2\in \mathcal{Q}_0(F_0^m)\atop \sigma(Q_1\cap Q_2)>0,z\in Q_1\cap Q_2}(a_{Q_1}-a_{Q_2})^2\\
&\leq 2\sum_{Q_1,Q_2\in \mathcal{Q}_0(F_0^m)\atop \sigma(Q_1\cap Q_2)>0,z\in Q_1\cap Q_2}\big((a_{Q_1}-b_{Q_1\cap Q_2})^2+(b_{Q_1\cap Q_2}-a_{Q_2})^2\big)\\
&=4\sum_{Q_1,Q_2\in \mathcal{Q}_0(F_0^m)\atop \sigma(Q_1\cap Q_2)>0,z\in Q_1\cap Q_2}(a_{Q_1}-b_{Q_1\cap Q_2})^2. 
\end{align*}

Finally, we see that \eqref{e:5.9} implies 
\[
\mcE_{D_m}|_{Q\cap D_m}(h)\leq c_1\sum_{Q'\in\mathcal{Q}_0(F_0^m):Q'\cap Q\neq\emptyset}\mcE_{G_m}|_{Q^{\prime}\cap G_m}(u_m^G)
\]
for some constant $c_1>0$. So by taking the summation over $Q\in \mathcal{Q}_0(F_0^m)$,
\begin{align*}
(R_m^D)^{-1}\leq\mcE_{D_m}(h)=\sum_{Q\in\mathcal{Q}_0(F_0^m)}\mcE_{D_m}|_{Q\cap D_m}(h)
\leq 8c_1 \mcE_{G_m}(u_m^G)=8c_1 (R_m^G)^{-1}
\end{align*}
gives $R_m^G\leq8c_1\,R_m^D$. 
\end{proof}

\subsection{Lower bound estimate of $R_n$}\label{section53}
Let $u_n$ be the optimum potential in \eqref{e:5.1}. Let $u_n^i,\ 1\leq i\leq 4$, be the rotation of $u_n$ by $\frac{(i-1)\pi}{2}$ counterclockwise. Define function $v_i,w_i$ on $T_i$ by
\[
v_i=u_n^{i+1}|_{T_i},\ w_i=u_n^i|_{T_i}.
\]
Clearly, $v_i(C_i)=v_i(C_{i+1})=0$; $w_i(C_i)=0,\ w_i(C_{i+1})=1$; $v_i(\overline{C})=w_i(\overline{C})=\frac{1}{2}$; and $R_n^{-1}=2\tmcE_0|_{T_i}(v_i)+2\tmcE_0|_{T_i}(u_i)$. The following lemma is same as Lemma 4.2 in \cite{barlow1990resistance}.

\begin{lem}$\tmcE_0|_{T_i}(v_i,w_i)=0,\ 1\leq i\leq 4$.
\end{lem}

For a function $\psi$ on $X_n^{(0)}\cup Z_n^{(0)}$ and harmonic at the center, let $z_i=\psi(C_i),\ i=1,2,3,4$ and $\overline{z}=\frac{1}{4}\sum_{i=1}^4 z_i$. Define $f_n^{\psi}:F_0^n\rightarrow \R$ by
\begin{equation}\label{e:5.10}
f_n^{\psi}=\sum_{i=1}^4[z_i+(2\overline{z}-z_i-z_{i+1})v_i+(z_{i+1}-z_i)w_i]\mathds{1}_{T_i}.
\end{equation}
By symmetry, it is easy to check that $f_n^{\psi}$ agrees with $\psi$ on $X_0^{(n)}\cup Z_0^{(n)}$, and can be continuously extended to $\bigcup_{i=1}^4C_i \overline{C}$ so that $f_n^{\psi}\in W^{1,2}(F_0^n)$.  Then 
\begin{equation}\label{e:5.11}
\begin{aligned}
\tmcE{_0}|_{F_0^n}(f_n^{\psi})&=\sum_{i=1}^4[((z_{i+1}-\overline{z})+(z_{i}-\overline{z}))^2\tmcE_0|_{T_i}(v_i)+(z_{i+1}-z_i)^2\tmcE_0|_{T_i}(w_i)]\\
&\leq 2 R_n^{-1}\sum_{i=1}^4(z_i-\bar{z})^2.
\end{aligned}
\end{equation}

\begin{lem}\label{lemma5.10} 
	$\frac{1}{2}R_n R_m^D\leq R_{n+m}$ for $n,m\geq 0$.
\end{lem}
\begin{proof}
Let $u_m^D$ be the optimum potential for $R_m^D$.
Define ${\psi}:3^n D_m\rightarrow \R$ by 
\[
{\psi}(3^n a)=u_m^D(a)\quad \hbox{ for } a\in  D_m.
\]
For $Q\in \mathcal{Q}_{n}(F_0^{n+m})$, let $\psi_Q$ be the restriction of $\psi$  on $Q\cap 3^nD_m,$ 
and define
\[ 
f=\sum_{Q\in \mathcal{Q}_{n}(F_0^{n+m})}f_n^{\psi_Q\circ\Phi_{n,Q}}\circ\Phi_{n,Q}^{-1}.
\]
It is direct to check that $f=1\hbox{ on } L_2,\ f=0\hbox{ on }L_4$, so $f$ is feasible in the sense of  \eqref{e:5.1}. Then using \eqref{e:5.11}, we have
\begin{align*}
R_{n+m}^{-1}&\leq \tmcE{_0}|_{F_0^{n+m}}(f)=\sum_{Q\in \mathcal{Q}_{n}(F_0^{n+m})}\tmcE{_0}|_{Q\cap F_0^{n+m}}(f_n^{\psi_Q\circ\Phi_{n,Q}}\circ\Phi_{n,Q}^{-1})\\
&\leq 2R_n^{-1}\sum_{Q\in \mathcal{Q}_{0}(F_0^{m})} \mcE_{D_m}|_{Q\cap D_m}(u_m^D)\\
&=2R_n^{-1}(R_m^D)^{-1}.
\end{align*}
\end{proof}
Finally, we turn to prove Theorem \ref{Theorem 5.1.}.

\begin{proof}[Proof of Theorem \ref{Theorem 5.1.}]
By Lemmas \ref{lemma5.7}, \ref{lemma5.8} and \ref{lemma5.10}, there exist a positive constant $c$ such that 
\begin{equation*}
    c^{-1}R_mR_n\leq R_{n+m}\leq cR_{m}R_n \hbox{ for } n,m\geq0.
\end{equation*}
Then $\log(c^{-1}R_n)$ is superadditive, $\log(cR_n)$ is subadditive. By Fekete lemma, 
\begin{align*}
&\lim_{n\rightarrow\infty} \frac{\log(cR_n)}{n}=\inf_{n\geq0}  \frac{\log(cR_n)}{n},\\
&\lim_{n\rightarrow\infty} \frac{\log(c^{-1}R_n)}{n}=\sup_{n\geq0}  \frac{\log(c^{-1}R_n)}{n},
\end{align*}
which yield Theorem \ref{Theorem 5.1.} by setting $\lambda=\exp(\lim\limits_{n\rightarrow\infty}\frac{\log R_n}{n})$.
\end{proof}

\section{Scaling limits}\label{section6}
Let $(\tmcE_0,\tmcF_0)$ be the Dirichlet form on $\tilde{F}_0$. Let $\Omega$ be an open set in $\tilde{F}_0$ and $A$ be a pre-compact set in $\Omega$. Define $\res(A,\Omega)$ by 
\begin{equation*}
\hbox{res}(A,\Omega):=(\inf\{\tmcE_0(f):\ f=1\ \hbox{ on $A$},\ f\in C_{c}(\Omega)\cap\tmcF_0 \})^{-1}.    
\end{equation*}
It follows from the definition that $\res(A,\Omega)$ is decreasing in $A$ and increasing in $\Omega$.

 For simplicity, in this section, we consider the infinity distance between points in $\R^2$ from time to time, 
\[
\|x-y\|_\infty:=\max\{|x_1-y_1|,|x_2-y_2|\}\quad\hbox{ for }x=(x_1,x_2)\hbox{ and }y=(y_1,y_2).
\]
Note that $\|x-y\|_\infty\asymp |x-y|$, where $|x-y|$ is the usual Euclidean distance between $x\hbox{ and }y$. For $x\in\tilde{F}_0$ and $r>0$, we denote 
\begin{align*}
  & B_\infty(x,r):=\{y\in \tilde{F}_0:\|y-x\|_\infty<r\},\\
  & B(x,r):=\{y\in \tilde{F}_0:|y-x|<r\}.
\end{align*}

\begin{lem}\label{lemma6.1}
For each $k\geq 1$, there exists a constant $C(k)>1$ such that
\begin{equation*}
\begin{aligned}
    C(k)^{-1} \frac{R_n\,{\mu_0(F^n_0)}}{\mu_0(B_\infty(x_0,3^n))}\leq \res({B_\infty(x_0,k3^n)},B_\infty(x_0,(k+1)3^n))\leq C(k) \frac{R_n\,{\mu_0(F^n_0)}}{\mu_0(B_\infty(x_0,3^n))}
\end{aligned}
\end{equation*}
for every $n\in \Z$ and $x_0=(i3^{n},j3^{n})$ with $i,j\in \mathbb{Z},\ i,j\geq0$. 
\end{lem}
\begin{proof}
Let $B_1=B_\infty(x_0,k3^n)$ and $B_2=B_\infty(x_0,(k+1)3^n)$. We prove the lower bound estimate of $\res(B_1,B_2)$ by constructing a function $f$ feasible for $\res({B_1},B_2)$. Let $S_1,...,S_{8k+4}$ be the squares in
\[
\mathcal{Q}_{n}([(i-k-1)3^n,(i+k+1)3^n]\times [(j-k-1)3^n,(j+k+1)3^n]\setminus \overline{B_1}).
\]
We label counterclockwise $S_l,\ l=1,...,8k+4$ starting with 
\[
S_1=[(i+k)3^n,(i+k+1)3^n]\times[j3^n,(j+1)3^n].
\] 
For $l\in\{1,...,2k+1\}$, we write $\varphi_l$ for the translation that maps $[0,3^n]^2$ to $S_l$. Let $u$ be the optimum function for \eqref{e:5.1}, and rotate $u$ to get $v,w$ such that 
\begin{equation}
\begin{aligned}\label{1}
  &v=1\hbox{ on }L_4,\quad\ v=0\hbox{ on }L_2,\quad\, \tmcE_0|_{F_0^n}(v)=1/R_n,\\
&w=1\hbox{ on }L_1,\quad w=0\hbox{ on }L_3,\quad \tmcE_0|_{F_0^n}(w)=1/R_n.   
\end{aligned}
\end{equation}
Note that $v\wedge w=\frac{1}{2}(v+w-|v-w|)$, by the Markov property of $\tmcE_0$, 
\begin{align}\label{2}
\tmcE_0(v\wedge w)\leq \tmcE_0(v)+\tmcE_0(w)=2/R_n. 
\end{align}
For $1\leq l\leq 2k+1$, we set $f_l:S_l\cap \tilde{F}_0\rightarrow \R$ by 
\[f_l=\begin{cases}
v\circ \varphi_l^{-1},&\qquad\hbox{ if  }l=1,...,k,\\
(v\wedge w)\circ \varphi_l^{-1}, &\qquad\hbox{ if  }l=k+1,\\
w\circ \varphi_l^{-1},&\qquad\hbox{ if  }l=k+2,...2k+1.
\end{cases} 
\]
For $l=2k+2,...8k+4$, we choose a reflection $g$ that maps $[(i-k-1)3^n,(i+k+1)3^n]\times [(j-k-1)3^n,(j+k+1)3^n]$ onto itself such that $g(S_{l})=S_{l^{\prime}}$ for some $l^{\prime}\in \{1,...,2k+1\}$, and set $f_l= f_l^{\prime}\circ g$ on $S_l\cap\tilde{F}_0$. Define $f:{ \tilde{F}_0}\rightarrow \R$ as 
\[
f=\begin{cases}
    f_l \qquad  &\hbox{on }S_l,\ \hbox{if }S_l\subset\tilde{F}_0,\\
    1 &\hbox{on }\overline{B_1},\\
    0 &\hbox{otherwise}.
\end{cases}
\]
Let $c_k=\max\{\frac{\mu_0(S_l\cap\tilde{F}_0)}{\mu_0(B_\infty(x_0,3^n))}:l=1,...8k+4\}$. 
 Then
\begin{align*}
\tmcE_0(f)&=\sum_{l=1}^{8k+4}\tmcE_0|_{S_l\cap \tilde{F}_0}(f)\\ &\leq\sum_{l= k+1,3k+2\atop  5k+3,7k+4}2\frac{\mu_0(S_l\cap\tilde{F}_0)}{\mu_0(F^n_0)}R^{-1}_n+\sum_{\substack{1\leq l\leq 8k+4, \\l\neq  k+1,3k+2\\
 5k+3,7k+4}}\frac{\mu_0(S_l\cap\tilde{F}_0)}{\mu_0(F^n_0)}R^{-1}_n\\
 &\leq (8k+8)c_k\frac{\mu_0(B_\infty(x_0,3^n))}{\mu_0(F_0^n)}R^{-1}_n,
\end{align*}
where the first inequality follows from \eqref{1} and \eqref{2}. This implies the desired lower bound of $\res(B_1,B_2)$.

Conversely, by the connectedness of $\tilde{F}_0$, we can pick a square among $\{S_1,S_2,\cdots,S_{8k+4}\}\setminus\{S_{k+1},S_{3k+2},S_{5k+3},S_{7k+4}\}$, say $S$, such that $S\subset \tilde{F}_0$. Then for any $h$ feasible for $\res(B_1,B_2)$, 
\begin{equation*}
\begin{aligned}
\tmcE_0(h)\geq \tmcE_0|_{S\cap\tilde{F}_0}(h|_{S})\geq c_k^{\prime} \frac{\mu_0(B_\infty(x_0,3^n))}{\mu_0(F_0^n)R_n},
\end{aligned}
\end{equation*}
where $c_k^{\prime}=\inf\{\frac{\mu_0(S_l\cap { \tilde{F}_0})}{\mu_0(B_\infty(x_0,3^n))}:l=1,...8k+4{,\ S_l\subset \tilde{F}_0}\}$. The upper bound estimate follows.
\end{proof}

Recall that $\mu_0(F_0^n)=(4+4\rho)^n$  and $R_n\asymp \lambda^n$ by Theorem \ref{Theorem 5.1.}, for $n\geq 1$. We let 
\[
\beta=\frac{\log(4+4\rho)}{\log3}+\frac{\log \lambda}{\log 3}\quad\hbox{ and }\quad \Psi(r)=
\begin{cases} 
	r^\beta, & \text{if }r\geq1,\\  
	r^2, & \text{if } 0<r\leq 1.
\end{cases}
\]

\begin{lem}\label{lemma6.2}
There exists a constant $C>1$ such that
\[
C^{-1}\Psi(r)/\mu_0(B_\infty(x,r))\leq \res(B_\infty(x,r),B_\infty(x,4r))\leq C \Psi(r)/\mu_0(B_\infty(x,r))
\]
for every $x\in \tilde{F}_0$ and $r>0$. 
\end{lem}
\begin{proof}
Choose $n\in\Z$ such that $3^{n+1}<r\leq 3^{n+2}$, and let $x_0$ be the center of $D_n(x)$. Note that $\|x-x_0\|_\infty\leq 3^n/2$, so  
\begin{equation*}
    B_\infty(x_0,3^n)\subset B_\infty(x,r)\subset B_\infty(x_0,10\cdot 3^{n})\subset  B_\infty(x_0,11\cdot 3^{n})\subset B_\infty (x,4r)\subset B_\infty(x_0,37\cdot 3^n).
\end{equation*}
By the monotonicity of $\res$, we have 
\begin{align*}
\res({B_\infty(x_0,10\cdot3^{n})},B_\infty(x_0,11\cdot3^{n}))&\leq\res({B_{\infty}(x,r)},B_{\infty}(x,4r))\\&\leq \res({B_{\infty}(x_0,3^n)},B_\infty(x_0,37\cdot3^n)).
\end{align*}
For the lower bound estimate, by Lemma \ref{lemma6.1} and Theorem \ref{Theorem 5.1.}, 
\begin{equation*}
    \begin{aligned}
        \res({B_\infty(x_0,10\cdot3^{n})},B_\infty(x_0,11\cdot3^{n}))&\gtrsim  \Psi(3^n)/\mu_0(B_\infty(x_0,3^n))\\
 &\gtrsim  \Psi(3^{n+2})/\mu_0(B_\infty(x_0,3^{n})) \\
 &\geq \Psi(r)/\mu_0(B_\infty(x,r)).       
    \end{aligned}
\end{equation*}
For the upper bound estimate, we observe that 
\begin{equation}\label{e:6.5}
\res(A_1,A_2)+\res(A_2,A_3)\leq \res(A_1,A_3)\quad\hbox{ for }A_1\subset A_2\subset A_3\subset \tilde{F}_0
\end{equation}
by \cite[Lemma 2.5]{grigor2004eigenvalues}, so   
\begin{align*}
\res({B_\infty(x,r)},B_\infty(x,4r))&\leq \res({B_{\infty}(x_0,3^n)},B_\infty(x_0,37\cdot3^n)\\
&\leq \sum_{i=1}^{36}\res({B_\infty(x_0,i3^n)},B_{\infty}(x_0,(i+1) 3^{n}))\\
&\lesssim\sum_{i=1}^{36} \Psi(3^n)/\mu_0(B_{\infty}(x_0,3^n))\\
&\lesssim \sum_{i=1}^{36} \Psi(r)/\mu_0(B_\infty(x_0,r))\\
&\lesssim \Psi(r)/\mu_0(B_\infty(x_0,r)).
\end{align*}
\end{proof}

By Theorem \ref{theorem4.1}, Lemma \ref{lemma6.2} and \cite[Theorem 3.14]{grigor2014heat}, the transition density $q_t(x,y)$ of the reflected Brownian motion $W$ on $\tilde{F}_0$ exists, and has a jointly H\"older continuous version. Moreover, there exist constants $C_1,C_2,C_3,C_4,C_5>0$ and $0<C_6<1$ such that:\smallskip 

\noindent$(a)$. for $t\geq 1$  and $x,y\in\tilde{F}_0$ such that $|x-y|\leq t$,
\begin{equation}\label{e:6.6}
 \begin{aligned}
 C_1 \frac{1}{\mu_0(B(x,t^{1/\beta}))}&\exp(-C_2 (\frac{|x-y|^\beta}{t})^{\frac1{\beta-1}})\\
 &\leq q_t(x,y)
 \leq C_3 \frac{1}{\mu_0(B(x,t^{1/\beta}))}\exp(-C_4 (\frac{|x-y|^\beta}{t})^{\frac1{\beta-1}})),     
 \end{aligned}   
\end{equation}

\noindent$(b)$. for $x\in\tilde{F}_0$ and $r>0$,
\begin{equation}\label{e:6.7}
    C_5^{-1}\Psi(r)\leq E^y[\tau_{B(x,r)}]\leq C_5\Psi(r)\quad\hbox{ for }y\in B(x,C_6r).\medskip 
\end{equation}

For $n\geq 0$, we define a process $X^{(n)}$ on $\tilde{F}_n$ as 
\[
X_t^{(n)}=3^{-n}W_{3^{n\beta}t}\quad\hbox{ for } t\geq 0.
\]
Write $P_n^x $ for the probability law of $X^{(n)}$ started at $x$. As an immediate consequence of \eqref{e:6.7}, 
\begin{equation}\label{e:6.4}
C_5^{-1}r^\beta\leq E_n^y[\tau_{B(x,r)}]\leq C_5r^\beta\quad\hbox{ for }x\in\tilde{F}_n,\,y\in B(x,C_6r)\cap\tilde{F}_n\hbox{ and }r>3^{-n}.
\end{equation}
Let $q_t^{(n)}$ be the transition density of $X^{(n)}$ with respect to $\mu_n$. Then, for $t>0$ and $x\in\tilde{F}_n$,  
\begin{equation*}
	\begin{aligned}
		P_n^x(X_t^{(n)}\in A)&=P^{3^nx}(3^{-n}W_{3^{n\beta}t}\in A)\\&=\int_{3^nA}q_{3^{n\beta}t}(3^nx,y)\mu_0(dy)=({ 4\rho+4})^n\int_Aq_{3^{n\beta}t}(3^nx,3^ny)\mu_0(dy).
	\end{aligned}
\end{equation*}
Hence, 
\begin{equation*}
q_t^{(n)}(x,y)=({ 4\rho+4})^n q_{3^{n\beta}t}(3^nx,3^ny).
\end{equation*}
The above observation, together with \eqref{e:6.6} and the fact that $\mu_n(B(x,r))=(4\rho+4)^{-n}\mu_0(B(3^nx,3^nr))$, implies
\begin{equation}\label{e:6.11}
 \begin{aligned}
 C_1 \frac{1}{\mu_n(B(x,t^{1/\beta}))}&\exp(-C_2 (\frac{|x-y|^\beta}{t})^{\frac1{\beta-1}})\\
 &\leq q_t^{(n)}(x,y)
 \leq C_3 \frac{1}{\mu_n(B(x,t^{1/\beta}))}\exp(-C_4 (\frac{|x-y|^\beta}{t})^{\frac1{\beta-1}}))
 \end{aligned}  
\end{equation}
for $t>0$ and $x,y\in\tilde{F}_n$ such that $3^{n\beta}t\geq (1\vee 3^n|x-y|)$.

Finally, we construct a limit process $X_t$ as the subsequential limit of $X^{(n)}_t$.  \medskip 

\begin{thm}
	There is a Feller process $((X_t)_{t>0},(P^x)_{x\in \tilde{F}})$ on $\tilde{F}$ such that there exists a transition density $p_t(x,y)$ satisfying the two-sided sub-Gaussian heat kernel estimate 
	\begin{equation}\label{e:6.22}
	\begin{split} 
		C_1\frac{1}{\mu(B(x,t^{1/\beta}))}\exp(-C_2(\frac{|x-y|^{\beta}}{t})^{\frac{1}{\beta-1}})&\leq p_t(x,y)\\&\leq C_3\frac{1}{\mu(B(x,t^{1/\beta}))}\exp(-C_4(\frac{|x-y|^{\beta}}{t})^{\frac{1}{\beta-1}})
	\end{split} 
	\end{equation}
	for $x,y\in\tilde{F}$ and $t>0$. Here, $C_1$--$C_4$ are the same constants in \eqref{e:6.6}. 
	\begin{proof}
		The proof is essentially the same as the construction of Barlow and Bass \cite{barlow1989construction,BB19993d}. We briefly review the key steps here.
		
		First, by \eqref{e:6.4} and the same proof of \cite[Proposition 5.5(c)]{BB19993d},  
		\[
		P_n^x(\sup_{0<s<t}|X^{(n)}_t-x|\geq r)\leq C_7\exp(-C_8(\frac{r^{\beta}}{t})^{\frac{1}{\beta-1}})\quad\hbox{ for } x\in\tilde{F}_n\hbox{ and }r\geq3^{-n},
		\]
for some constants $C_7,C_8>0$, and hence the law of $P_n^{x_n},n\geq1$ is tight for a bounded sequence $x_n\in\tilde{F}_n$ by the Aldous criterion for tightness  (see \cite[Theorem 16.10 and 16.11]{Kallenberg} for example). 
		
		Next, by the elliptic Harnack inequality (Theorem \ref{theorem4.1}), the mean exit time estimate \eqref{e:6.7} and the same proof of \cite[Section 3.1]{BBKT}, we have 
		$U_\lambda^{(n)}f\hbox{ is equicontinuous for }f\in C_c(\tilde{F}_n)$,
		where $U_\lambda^{(n)}$ is the resolvent operator associated with $X^{(n)}$, i.e. 
		\[
		U_\lambda^{(n)}f(x)=E_n^x[\int_0^{\infty}e^{-\lambda t}f(X_t^{(n)})dt].
		\]
		Hence, by a diagonal argument, there is a subsequence $n_k$ so that $U_\lambda^{(n_k)}f$ converges uniformly on $\tilde{F}$ for every $f\in C_c(\tilde{F}_0)$. 
		
		Finally, by the same argument of \cite[Section 6]{barlow1989construction}, there is a $\mu$-symmetric Feller process $X$ on $\tilde{F}$ such that $P^{x_k}_{n_k}((X^{(n_k)}_t)_{t\geq 0}\in \bullet)$ weakly converges to $P^x((X_t)_{t\geq 0}\in \bullet)$ for every $x_k\in \tilde{F}_{n_k}$ and $x\in\tilde{F}$ such that $x_k\to x$. The transition density estimate \eqref{e:6.22}  follows from \eqref{e:6.11} by taking the limit. \end{proof}
	\end{thm}	
		
		 Finally, noting that the Sierpi\'nski carpet $F\setminus (F\cap (\overline{\tilde{F}\setminus F}))$ is a uniform domain, by a same proof of \cite{Lierl}, the desired heat kernel estimate \eqref{e:1.1} holds on $F$ for the reflected diffusion process by \cite[Theorem 2.8]{MR4797375}. Hence, we have proved Theorem \ref{thm1}.

\end{document}